\documentclass[11pt,a4paper]{article}
\usepackage{amsmath}
\usepackage[a4paper]{geometry}
\usepackage{amssymb,latexsym,amsmath,amsfonts,amsthm}
\usepackage{graphicx}
\usepackage{algorithm}
\usepackage{algorithmic}
 \usepackage{dsfont}
\usepackage{mathrsfs}
\usepackage{epsfig}
\usepackage{booktabs}
\usepackage{mathtools}
\usepackage{comment,verbatim}
\usepackage{overpic}
\usepackage{hyperref}
\usepackage{bbm}
\usepackage[toc,page]{appendix}

\renewcommand{\Re}{\mathrm{Re}\,}

\newtheorem{theorem}{Theorem}[section]
\newtheorem{lemma}[theorem]{Lemma}

\newtheorem{corollary}[theorem]{Corollary}

\theoremstyle{definition}

\theoremstyle{definition}
\newtheorem{example}[theorem]{Example}

\theoremstyle{remark}

\newtheorem{remark}[theorem]{Remark}

\numberwithin{equation}{section}

\usepackage{color}

\usepackage[normalem]{ulem}

\def\XXint#1#2#3{{
\setbox0=\hbox{$#1{#2#3}{\int}$}
\vcenter{\hbox{$#2#3$}}\kern-.5\wd0}}

\DeclarePairedDelimiter\floor{\lfloor}{\rfloor}

\begin{document}
\title{On the optimal estimates and comparison of Gegenbauer expansion coefficients\footnotemark[1]}
\author{Haiyong Wang\footnotemark[2]
\\[3\jot]
School of Mathematics and Statistics \\
Huazhong University of Science and
Technology \\
Wuhan 430074, P. R. China}

\maketitle
\renewcommand{\thefootnote}{\fnsymbol{footnote}}

\footnotetext[1]{This work was supported by the National Natural
Science Foundation of China under grant 11301200 and the Fundamental
Research Funds for the Central Universities under grant 2015TS115.}

\footnotetext[2]{ E-mail: \texttt{haiyongwang@hust.edu.cn} }

\begin{abstract}
In this paper, we study optimal estimates and comparison of the
coefficients in the Gegenbauer series expansion. We propose an
alternative derivation of the contour integral representation of the
Gegenbauer expansion coefficients which was recently derived by
Cantero and Iserles [{\it SIAM J. Numer. Anal., 50 (2012),
pp.307--327}]. With this representation, we show that optimal
estimates for the Gegenbauer expansion coefficients can be derived,
which in particular includes Legendre coefficients as a special
case. Further, we apply these estimates to establish some rigorous
and computable bounds for the truncated Gegenbauer series. In
addition, we compare the decay rates of the Chebyshev and Legendre
coefficients. For functions whose singularity is outside or at the
endpoints of the expansion interval, asymptotic behaviour of the
ratio of the $n$th Legendre coefficient to the $n$th Chebyshev
coefficient is given, which provides us an illuminating insight for
the comparison of the spectral methods based on Legendre and
Chebyshev expansions.
\end{abstract}


\vspace{0.05in}


\section{Introduction}\label{sec:introduction}
Gegenbauer polynomials together with their special cases like
Legendre and Chebyshev polynomials are widely used in many branches
of numerical analysis such as interpolation and approximation
theories, the construction of quadrature formulas, the resolution of
Gibbs phenomenon and spectral and pseudo-spectral methods for
ordinary and partial differential equations. One of the most
attractive features is that these families of orthogonal polynomials
have excellent error properties in the approximation of a globally
smooth function. Typically, the error of the truncated series in
Gegenbauer polynomials decreases exponentially fast for analytic
functions as the number of the series increases. For entire
functions, the error will decrease even faster than exponential.
This remarkable property explains why Gegenbauer and its special
cases are extensively used to solve various problems arising from
science and engineering.

Let $C_k^{(\lambda)}(x)$ denote the Gegenbauer polynomial of degree
$k$ which is normalized by the following condition
\begin{equation}\label{eq:normalization gegenbauer}
C_{k}^{(\lambda)}(1) =
\frac{\Gamma(k+2\lambda)}{k!\Gamma(2\lambda)},
\end{equation}
where $\lambda > -\frac{1}{2}$ and $\lambda \neq 0$. For the special
case $\lambda = 0$, we have
\[
C_k^{(0)}(1) = \frac{2}{k}, \quad k \geq 1,
\]
and $ C_0^{(0)}(x) = 1$. For a fixed $\lambda$, these Gegenbauer
polynomials are orthogonal over the interval $[-1,1]$ with respect
to the weight function $\omega(x) = (1-x^2)^{\lambda-\frac{1}{2}}$
and
\begin{equation}
\int_{-1}^{1} (1-x^2)^{\lambda-\frac{1}{2}} C_{m}^{(\lambda)}(x)
C_{n}^{(\lambda)}(x) dx = h_n^{(\lambda)}  \delta_{mn},
\end{equation}
where $\delta_{mn}$ is the Kronnecker delta and
\begin{align}\label{eq:normalization constant}
h_n^{(\lambda)} = \frac{2^{1-2\lambda} \pi}{ \Gamma(\lambda)^2 }
\frac{\Gamma(n+2\lambda)}{\Gamma(n+1) (n+\lambda)}, \quad
\lambda\neq 0.
\end{align}
For applications in numerical analysis, it is often required to
expand a given function $f(x)$ in terms of Gegenbauer polynomials as
\begin{align}\label{eq:gegenbauer expansion}
f(x) = \sum_{n = 0}^{\infty} a_n^{\lambda} C_n^{(\lambda)}(x),
\end{align}
where the {\it Gegenbauer coefficients} are defined by
\begin{align}\label{eq:gegenbauer coefficients}
a_n^{\lambda} = \frac{1}{h_n^{(\lambda)} } \int_{-1}^{1}
(1-x^2)^{\lambda-\frac{1}{2}} f(x) C_n^{(\lambda)}(x) dx, \quad  n
\geq 0.
\end{align}
The Gegenbauer expansion \eqref{eq:gegenbauer expansion} is an
invaluable and powerful tool in a wide range of practical
applications. In particular, they are widely used in the resolution
of Gibbs phenomenon and numerical solutions of ordinary and partial
differential equations (see, for example,
\cite{boyd2005trouble,gottlieb1995gibbs1,gottlieb1995gibbs2,gottlieb1996gibbs,gottlieb1997gibbs,gottlieb1992gibbs,ben1999gegenbauer,olver2013fast}).
In these applications, it is frequently required to estimate the
error bound of the truncated Gegenbauer expansion in the uniform
norm. It is well known that the error of the truncated spectral
expansion depends solely on how fast the corresponding spectral
expansion coefficients decay. Therefore, the final aim is reduced to
the estimate of decay rate of the corresponding spectral
coefficients.

For the Gegenbauer coefficients, the estimate of their decay rate
was received attention in the past few decades and some results have
been developed in the literatures. For example, to defeat the Gibbs
phenomenon, Gottlieb and Shu in a number of papers
\cite{gottlieb1995gibbs1,gottlieb1995gibbs2,gottlieb1996gibbs,gottlieb1997gibbs,gottlieb1992gibbs}
proposed to reexpand the truncated Fourier or Chebyshev series into
a Gegenbauer series. To prove the exponential convergence of the
Gegenbauer series, a rough estimate of the Gegenbauer coefficients
was proposed (see \cite[Eqn.~(4.3)]{gottlieb1992gibbs}). More
recently, this issue was considered in
\cite{xiang2012error,zhao2013sharp} and some sharper estimates were
given. In these works, the main idea is to express the Gegenbauer
coefficient $a_n^{\lambda}$ as an infinite series either by using
the Chebyshev expansion of the first kind of $f$ or the Cauchy
integral formula together with the generating function of the
Chebyshev polynomial of the second kind, and then estimate the
derived infinite series term by term. As we shall see, these results
are often overestimated. For special cases of the Gegenbauer
spectral expansion such as Chebyshev and Legendre expansions, the
estimates of their expansion coefficients have been studied
extensively in the past few decades (see
\cite{bernstein1912ordre,boyd2009large,elliott1964evaluation,elliott1974asymptotic,wang2012convergence,xiang2012error,zhao2013sharp}
and references therein).

In this article we will develop a novel approach to study the
estimate of the Gegenbauer coefficients. The starting point of our
analysis is the contour integral expression which was recently
derived by Cantero and Iserles in \cite{cantero2012rapid}. Their
idea was based on expressing the Gegenbauer coefficients in terms of
an infinite linear combination of derivatives at the origin and then
by an integral transform with a Gauss hypergeometric function as its
kernel. This kernel function converges too slowly to be
computationally useful. To remedy this, a hypergeometric
transformation was proposed to replace the original kernel by a new
one which converges rapidly. This, to the author's knowledge, is the
first result on the contour integral expression of the Gegenbauer
coefficients. However, as the authors admit in their paper
\cite{cantero2012rapid,iserles2011fast}, the proof of their
derivation is rather convoluted.

Here, we shall provide an alternative and simple approach to derive
the contour integral expression of the Gegenbauer coefficients. Our
idea is motivated by the connection formula between the Gegenbauer
and Chebyshev polynomials of the second kind, which was initiated in
\cite{wang2014fast}. We show that the contour integral expression of
the Gegenbauer coefficients can be derived by rearranging the
Chebyshev coefficients of the second kind. With the derived contour
integral expression in hand, we prove that optimal estimates for the
Gegenbauer coefficients can be obtained as direct consequences. To
the best of our knowledge, these are the first results of this kind
that are proved to be optimal. In contrast to existing studies,
numerical results indicate that our estimates are superior. Further,
we apply these estimates to establish some rigorous and computable
bounds for the truncated Gegenbauer series in the uniform norm.

Among the families of Gegenbauer spectral expansions, Chebyshev and
Legendre expansions are the most popular and important cases (see
\cite{atkinson2012spherical,olver2013fast,trefethen2013approximation}).
One particularly interesting question for these two expansions is
the comparison of the decay rates of the Legendre and Chebyshev
coefficients. This issue was considered by Lanczos in
\cite{lanczos1952introduction} and later by Fox and Parker in
\cite{fox1968chebyshev}, where these authors deduced that the $n$th
Chebyshev coefficient decays approximately $\sqrt{n\pi}/2$ faster
than its Legendre counterpart. Very recently, Boyd and Petschek in
\cite{boyd2013relationships} checked this issue carefully and showed
that this assertion is not true for some concrete examples. In this
work, with the help of the contour integral expression of Gegenbauer
coefficients, we improve significantly the observation stated in
\cite{boyd2013relationships} and present delicate results on the
comparison of the Legendre and Chebyshev coefficients. More
precisely, for functions whose singularity is outside the interval
$[-1,1]$, the asymptotic behaviour of the ratio $\gamma_n =
a_n^L/a_n^C$, where $a_n^L$ and $a_n^C$ denote the $n$th Legendre
and Chebyshev coefficient of $f$ respectively, is derived which
shows that the above assertion by Lanczos and Fox and Parker is
always false if the singularity is finite. We also extend our
results to functions with endpoint singularity and subtle results on
the asymptotic behaviour of $\gamma_n$ are given.

%Meanwhile, we show that the ``Lanczos-Fox-Parker'' proposition is
%always false if the absolute value of the singularity is finite.
%However, the Chebyshev coefficient $a_n^C$ decays
%$\mathcal{O}(\sqrt{n})$ faster than its Legendre counterpart. For
%functions with endpoint singularities of logarithmic and algebraic
%type, subtle results are also given; see Theorem \ref{thm:asymptotic
%gamma endpoint singularity}.

This paper is organized as follows. In the next section, we collect
some well known properties of Gegenbauer polynomials which will be
used in the later sections. In section \ref{sec:contour gegenbauer},
we provide an alternative way to derive the contour integral
expression of Gegenbauer coefficients. With the contour integral
expression obtained, in section \ref{sec:estimate gegenbauer
coefficients} we present optimal estimates of the Gegenbauer
coefficients and apply these to derive error bounds of the truncated
Gegenbauer expansion in the uniform norm. A comparison of the
Legendre and Chebyshev coefficients are discussed in section
\ref{sec:comparison}. Finally, in section \ref{sec:conclusion} we
give some final remarks .

%Our main contributions are as follows:
%\begin{itemize}
%\item We provide a new and simple approach to derive the contour
%integral expression of the Gegenbauer coefficients.
%
%\item We establish new and optimal estimates for the Gegenbauer
%coefficients and derive some rigorous error bounds for the truncated
%Gegenbauer expansion. The special case that $\lambda$ is
%proportional to the number of the truncated Gegenbauer series is
%also considered.
%
%\item We present a delicate result on the issue of the ``Lanczos-Fox-Parker''
%proposition. In particular, exact formulas for the ratio of the
%$n$-th Legendre coefficient to the $n$-th Chebyshev coefficient are
%derived for some model functions, which provide us an illuminating
%insight to the comparison of the Legendre and Chebyshev
%coefficients.
%
%%the which states that the ratio asymptotes to $\sqrt{n\pi}/2$ as $n$
%%tends to infinity.
%\end{itemize}

\section{Some properties of Gegenbauer polynomials}
In this section we will collect some well-known properties of
Gegenbauer polynomials which will be used in the subsequent
analysis. All these properties can be found in
\cite{szego1939orthogonal}.

Gegenbauer polynomials satisfy the following three-term recurrence
relation
\begin{align}
(n+1) C_{n+1}^{(\lambda)}(x) = 2x( n + \lambda) C_{n}^{(\lambda)}(x)
- (n + 2\lambda - 1) C_{n-1}^{(\lambda)}(x), \quad n\geq 1,
\end{align}
where $C_0^{(\lambda)}(x) = 1$ and $C_1^{(\lambda)}(x) = 2\lambda
x$. These polynomials also satisfy the following symmetry relations
\begin{align}
C_{n}^{(\lambda)}(x) = (-1)^n C_{n}^{(\lambda)}(-x), \quad n\geq 0,
\end{align}
which imply that $C_{n}^{(\lambda)}(x)$ is an even function for even
$n$ and an odd function for odd $n$. Moreover, Gegenbauer
polynomials satisfy the following inequality
\begin{align}\label{eq:gegenbauer inequality}
|C_{n}^{(\lambda)}(x)| \leq C_{n}^{(\lambda)}(1), \quad |x|\leq1,~~
\lambda>0,~~ n\geq 0.
\end{align}

Gegebauer polynomials include some important polynomials such as
Legendre and Chebyshev polynomials as special cases. More
specifically, we have
\begin{align}\label{eq:gegenbauer and legendre}
P_n(x) = C_n^{(\frac{1}{2})}(x),  \quad  U_n(x) = C_n^{(1)}(x),
\quad n\geq 0,
\end{align}
where $P_n(x)$ is the Legendre polynomial of degree $n$ and $U_n(x)$
is the Chebyshev polynomial of the second kind of degree $n$. When
$\lambda = 0$, the Gegenbauer polynomials reduce to the Chebyshev
polynomials of the first kind by the following definition
\begin{align}\label{eq:gegenbauer and chebyshev first}
\lim_{\lambda\rightarrow0^{+}} \lambda^{-1} C_{n}^{(\lambda)}(x) =
\frac{2}{n} T_n(x), \quad n \geq 1,
\end{align}
where $T_n(x)$ is the Chebyshev polynomial of the first kind of
degree $n$.

\section{A simple derivation of contour integral expression for Gegenbauer coefficients}\label{sec:contour gegenbauer}
In this section we shall present a new and simple derivation of the
contour integral expression of Gegenbauer coefficients. Our idea is
based on the connection formula between the Gegenbauer polynomial
and the Chebyshev polynomial of the second kind.

%Before we commence our analysis, we give a useful lemma to show the
%connection between two different families of Gegenbauer polynomials.
\begin{lemma}\label{lemma:conversion gegenbauer}
We have
\begin{align}
C_n^{(\lambda)}(x) = \sum_{k=0}^{\floor*{ \frac{n}{2} } } \frac{
(\lambda)_{n-k} (\lambda - \mu)_{k} (n+\mu-2k) }{ (\mu+1)_{n-k} k!
\mu} C_{n-2k}^{(\mu)}(x),
\end{align}
where $\floor{\cdot}$ denotes the integer part.
\end{lemma}
\begin{proof}
See \cite[p.~360]{andrews2000special}.
\end{proof}

Based on the above lemma, we are now able to derive the contour
integral expression of Gegenbauer coefficients. Before that, we
introduce the Bernstein ellipse $\mathcal{E}_{\rho}$ which is
defined by
\begin{equation}
 \mathcal{E}_{\rho} =  \left\{ z \in \mathbb{C} ~~\bigg| ~~z =
 \frac{1}{2} \big( \rho e^{i \theta} + \rho^{-1} e^{-i \theta} \big),~~ \rho >1,~  ~0 \leq
\theta  \leq   2\pi  \right\}.
\end{equation}
Throughout this paper, we define the positive direction of contour
integrals as the counterclockwise direction.
\begin{theorem}\label{thm:gegenbauer contour integral}
Suppose that $f$ is analytic inside and on the ellipse
$\mathcal{E}_{\rho}$, then for each $n \geq 0$,
\begin{align}\label{eq:contour gegenbauer coefficient}
a_n^{\lambda} & = \frac{ c_{n,\lambda} }{i \pi }
\oint_{\mathcal{E}_{\rho}} \frac{f(z)}{ (z \pm \sqrt{z^2 - 1})^{n+1}
} {}_2\mathrm{ F}_1\left[\begin{matrix} n + 1, & 1 - \lambda;
\\   n + \lambda + 1; \hspace{-1cm} &\end{matrix}  \frac{1}{( z \pm \sqrt{z^2 - 1}
)^{2}} \right]  dz,
\end{align}
where the sign is chosen so that $|z \pm \sqrt{z^2 - 1}|>1$ and
\begin{align}\label{eq:gegenbauer constant}
c_{n,\lambda} = \frac{\Gamma(\lambda)
\Gamma(n+1)}{\Gamma(n+\lambda)},
\end{align}
and $\Gamma(z)$ is the gamma function. The Gauss hypergeometric
function ${}_2 \mathrm{F}_1$ is defined by
\[
{}_2 \mathrm{F}_1 \left[\begin{matrix} a_1,~ a_2;& \\  b_1;
\end{matrix} \hspace{-.25cm} z \right] = \sum_{k = 0}^{\infty} \frac{ (a_1)_k
(a_2)_k }{ (b_1)_k } \frac{ z^k }{ k! },
\]
where $(z)_k$ denotes the Pochhammer symbol defined by $(z)_{k} =
(z)_{k-1} (z + k - 1)$ for $k \geq 1$ and $(z)_0 = 1$.
\end{theorem}
\begin{proof}
First, we expand the function in terms of the Chebyshev polynomials
of the second kind,
\begin{align}\label{eq:chebyshev expansion second}
f(x) = \sum_{j=0}^{\infty} b_j U_j(x),  \quad b_j = \frac{2}{\pi}
\int_{-1}^{1} \sqrt{1-x^2} f(x) U_j(x) dx.
\end{align}
Alternatively, the above Chebyshev coefficients can be expressed by
the following contour integral representation \cite{wang2014fast}
\begin{equation}\label{eq:chebyshev coefficients second}
b_j = \frac{1}{\pi i } \oint_{ \mathcal{E}_{\rho} } \frac{f(z) }{ (z
\pm \sqrt{z^2 - 1})^{j+1} } dz, \quad j\geq 0,
\end{equation}
where the sign is chosen so that $|z \pm \sqrt{z^2 - 1}|>1$.
Substituting the above Chebyshev expansion into the Gegenbauer
coefficients yields
\begin{align}\label{eq:gegenbauer and chebyshev}
a_n^{\lambda} & = \frac{1}{h_n^{(\lambda)}} \int_{-1}^{1}
(1-x^2)^{\lambda-\frac{1}{2}} \left( \sum_{j=0}^{\infty} b_j U_j(x)
\right) C_n^{(\lambda)}(x) dx
\nonumber \\
& = \sum_{m=0}^{\infty} b_{n+2m} \sigma_{n+2m, n}^{( \lambda )},
\end{align}
where we have used the parity of the Gegenbauer and Chebyshev
polynomials and
\begin{align}\label{eq:connection coefficients}
\sigma_{n+2m, n}^{( \lambda )} = \frac{1}{h_n^{(\lambda)}}
\int_{-1}^{1} (1-x^2)^{\lambda-\frac{1}{2}} U_{n+2m}(x)
C_n^{(\lambda)}(x) dx.
\end{align}
Note that $U_n(x)$ is a special case of Gegenbauer polynomial, e.g.,
$U_n(x) = C_n^{(1)}(x)$. From Lemma \ref{lemma:conversion
gegenbauer} we obtain
\begin{align*}
U_{n+2m}(x) = \sum_{k=0}^{\floor{ \frac{n+2m}{2} } } \frac{
(1)_{n+2m-k} (1 - \lambda)_{k} (n+2m+\lambda-2k) }{
(\lambda+1)_{n+2m-k} k! \lambda} C_{n+2m-2k}^{(\lambda)}(x).
\end{align*}
Substituting this into \eqref{eq:connection coefficients} and using
the orthogonality of Gegenbauer polynomials yields
\begin{align}\label{eq:connection coefficients formula}
\sigma_{n+2m, n}^{( \lambda )} &= \frac{ (1)_{n+m} (1 - \lambda)_{m}
(n+\lambda) }{ (\lambda+1)_{n+m} m! \lambda} \nonumber \\
& = \frac{\Gamma(\lambda) \Gamma(n+1)}{\Gamma(n+\lambda)}
\frac{(n+1)_m (1-\lambda)_m}{(n+\lambda+1)_m m!}.
\end{align}
Combining this with equation \eqref{eq:gegenbauer and chebyshev}
gives
\begin{align*}
a_n^{\lambda} & = \frac{\Gamma(\lambda)
\Gamma(n+1)}{\Gamma(n+\lambda)} \frac{1}{\pi i} \oint_{
\mathcal{E}_{\rho}} \frac{ f(z)}{(z \pm \sqrt{z^2 - 1})^{n+1}}
\sum_{m=0}^{\infty} \frac{ (n+1)_m (1-\lambda)_m }{ (n+\lambda+1)_m
m!} \frac{1}{(z \pm \sqrt{z^2 -
1})^{2m}} dz \nonumber \\
& = \frac{\Gamma(\lambda) \Gamma(n+1)}{\Gamma(n+\lambda)}
\frac{1}{\pi i} \oint_{ \mathcal{E}_{\rho}} \frac{ f(z)}{(z \pm
\sqrt{z^2 - 1})^{n+1}} {}_2\mathrm{ F}_1\left[\begin{matrix} n + 1,
& 1 - \lambda;
\\   n + \lambda + 1; \hspace{-1cm} &\end{matrix} \frac{1}{( z \pm \sqrt{z^2 - 1}
)^{2}} \right]  dz.
\end{align*}
This completes the proof.
%Note that $f(z)$ ia analytic in the
%neighborhood of the interval $[-1,1]$. Thus, the contour of the last
%equality can be deformed to any closed contour within and on which
%$f$ is analytic.
\end{proof}

The following corollaries are immediate consequences of Theorem
\ref{thm:gegenbauer contour integral}.
\begin{corollary}\label{corollary:integer lambda}
When $\lambda$ is a positive integer, the contour integral
expression of the Gegenbauer coefficients can be further represented
as
\begin{align}
a_n^{\lambda} = \frac{c_{n,\lambda}}{i \pi} \sum_{k=0}^{\lambda - 1}
\frac{ (n+1)_k (1-\lambda)_k }{ (n+\lambda+1)_k k! } \oint_{
\mathcal{\mathcal{E}_{\rho}} } \frac{f(z)}{ (z \pm \sqrt{z^2 -
1})^{n+2k+1} } dz, \quad  n \geq 0.
\end{align}
\begin{proof}
Note that the Gauss hypergeometric function on the right hand side
of \eqref{eq:contour gegenbauer coefficient} reduces to a finite sum
when $\lambda$ is a positive integer,
\[
{}_2\mathrm{ F}_1\left[\begin{matrix} n + 1,~ 1 - \lambda; &
\\   n + \lambda + 1;  &\end{matrix} \hspace{-.25cm}  \frac{1}{( z \pm \sqrt{z^2 - 1}
)^{2}} \right] = \sum_{k=0}^{\lambda-1} \frac{ (n+1)_k (1-\lambda)_k
}{ (n+\lambda+1)_k k! } \frac{1}{(z \pm \sqrt{z^2-1})^{2k}}.
\]
The desired result follows.
\end{proof}
\end{corollary}

%An interesting special case corresponding to $\lambda = \frac{1}{2}$
%is the Legendre coefficients. The next result gives the contour
%integral expression of the Legendre coefficients. A slightly
%different expression was first given by Iserles in
%\cite[Eqn.~(3.3)]{iserles2011fast}.
%\begin{corollary}
%The Legendre spectral expansion is defined by
%\begin{align}\label{eq:legendre expansion}
%f(x) = \sum_{n= 0}^{\infty} a_n^L P_n(x), \quad  a_n^L = \left(n +
%\frac{1}{2} \right) \int_{-1}^{1} f(x) P_n(x) dx.
%\end{align}
%Under the same assumptions as in Theorem \ref{thm:gegenbauer contour
%integral}, for each $n \geq 0$ we have
%\begin{align}\label{corollary:legendre contour expression}
%a_n^{L} =  \frac{c_{n,\frac{1}{2}}}{\pi i} \oint_{
%\mathcal{E}_{\rho} } \frac{ f(z)}{(z \pm \sqrt{z^2 - 1})^{n+1}}
%{}_2\mathrm{ F}_1\left[\begin{matrix} n + 1, & \frac{1}{2};
%\\ n + \frac{3}{2}; \hspace{-.7cm} &\end{matrix} \frac{1}{( z \pm \sqrt{z^2 - 1}
%)^{2}} \right] dz,
%\end{align}
%where $c_{n,\lambda}$ is defined as in \eqref{eq:gegenbauer
%constant}.
%\begin{proof}
%Note that $P_n(x) = C_n^{(\frac{1}{2})}(x)$, thus it is obvious that
%for $n \geq 0$ we have $a_n^{L} = a_n^{\frac{1}{2}}$. This completes
%the proof.
%\end{proof}
%\end{corollary}

The Cauchy transform of $\frac{1}{2} (1 - x^2)^{\lambda -
\frac{1}{2}} C_n^{(\lambda)}(x)$ is defined by
\begin{equation}\label{def:cauchy transform gegenbauer}
Q_n^{(\lambda)}(z) = \frac{1}{2} \int_{-1}^{1} \frac{ (1 -
x^2)^{\lambda - \frac{1}{2}} C_n^{(\lambda)}(x) }{z - x} dx, \quad
z\in \mathbb{C}\setminus [-1,1],
\end{equation}
which appears in the remainder of Gauss-Gegenbauer quadrature (see
\cite{gautschi1983error}). The proof of our above theorem leads to
an explicit form of this function which has not been found in the
literature. We state it in the following.
\begin{corollary}\label{corollary:gegenbauer function second}
For each $n \geq 0$, we have
\begin{align}
Q_n^{(\lambda)}(z) =  \frac{ c_{n,\lambda} h_n^{(\lambda)}}{ ( z \pm
\sqrt{z^2 - 1} )^{n+1} } {}_2\mathrm{ F}_1\left[\begin{matrix} n +
1,~ 1 - \lambda;&
\\   n + \lambda + 1;  &\end{matrix} \hspace{-.25cm} \frac{1}{( z \pm \sqrt{z^2 - 1}
)^{2}} \right],
\end{align}
where $c_{n,\lambda}$ and $h_n^{(\lambda)}$ are defined as in
\eqref{eq:gegenbauer constant} and \eqref{eq:normalization constant}
respectively and the sign is chosen such that $|z \pm \sqrt{z^2 -
1}| > 1$.
\begin{proof}
By setting $z = \frac{1}{2}(u + u^{-1})$ and thus $ u = z \pm
\sqrt{z^2 - 1} $ and the sign is chosen so that $|u|>1$. Recall the
generating function of $U_n(x)$ (see, e.g.,
\cite[Eqn.~(18.12.10)]{olver2010nist}), we obtain
\begin{align}
\frac{1}{z - x} &=  \frac{2}{u} \frac{1}{1 - 2xu^{-1} + u^{-2}} \nonumber \\
&= 2 \sum_{k=0}^{\infty} \frac{U_k(x)}{u^{k+1}}.
\end{align}
Substituting this into the integral expression of
$Q_n^{(\lambda)}(z)$ and using \eqref{eq:connection coefficients}
and \eqref{eq:connection coefficients formula} gives the desired
result.
\end{proof}
\end{corollary}

\begin{remark}\label{remark:asymptotic of c}
An interesting question, which is relevant for our subsequent
analysis, is the asymptotic behaviour of the constant
$c_{n,\lambda}$ as $n\rightarrow\infty$. By using the asymptotic
expansion of the ratio of gamma functions
\cite[C.4.3]{andrews2000special}
\begin{align}\label{eq:asymptotic ratio gamma function}
\frac{\Gamma(z+a)}{\Gamma(z+b)} \sim z^{a-b} \left[ 1 +
\frac{(a-b)(a+b-1)}{2z} + \mathcal{O}(z^{-2}) \right], \quad
z\rightarrow\infty,
\end{align}
we have, for fixed $\lambda$, that
\begin{align}\label{eq:asymptotic of c}
c_{n,\lambda} & = \frac{\Gamma(\lambda)
\Gamma(n+1)}{\Gamma(n+\lambda)}   \nonumber \\
& = \Gamma(\lambda) n^{1-\lambda} \left( 1 + \frac{\lambda (1 -
\lambda)}{2n} + \mathcal{O}(n^{-2}) \right) , \quad
n\rightarrow\infty.
\end{align}
Clearly, when $-\frac{1}{2} < \lambda < 1$ and $\lambda \neq 0$, the
constant $c_{n,\lambda}$ grows algebraically as $n$ grows. For
$\lambda > 1$, however, it decays algebraically as $n$ grows. On the
other hand, the situation will be greatly different if $\lambda$
varies with respect to $n$. For example, to remove the Gibbs
phenomenon, it was proposed to employ the truncated Gegenbauer
expansion by choosing $\lambda = \alpha n$ for some constant
$\alpha>0$ (see
\cite{gottlieb1995gibbs1,gottlieb1995gibbs2,gottlieb1996gibbs,gottlieb1997gibbs,gottlieb1992gibbs}
for more details). In this case, by using the asymptotic behaviour
of the gamma function, we find
\begin{align}
c_{n,\lambda} & = \left(
\frac{\alpha^{\alpha}}{(\alpha+1)^{\alpha+1}} \right)^{n} \left(
\sqrt{2\pi n \left(\frac{\alpha+1}{\alpha} \right) } +
\mathcal{O}(n^{-\frac{1}{2}} ) \right), \quad  n\rightarrow\infty.
\end{align}
Note that
\[
0< \frac{\alpha^{\alpha}}{(\alpha+1)^{\alpha+1}} < 1, \quad \alpha >
0,
\]
this implies that $c_{n,\lambda}$ decays exponentially as $n$ grows.
Meanwhile, we point out that the behaviour of $c_{n,\lambda}$ only
depends on $n$ and $\lambda$ but not on the function $f$.
\end{remark}

\section{Optimal estimates for the Gegenbauer coefficients and error bounds for the truncated Gegenbauer
expansion}\label{sec:estimate gegenbauer coefficients} An important
application of the contour integral expression is that it can be
used to establish some rigorous bounds on the rate of decay of
Gegenbauer coefficients. In this section, we shall establish optimal
and computable estimates for the Gegenbauer coefficients. Comparing
with existing studies, we show that our results are sharper.
Further, we apply these estimates to establish some error bounds of
the truncated Gegenbauer expansion in the uniform norm.
\subsection{Optimal estimates for the Gegenbauer coefficients}
The following theorem gives computable estimates for the Gegenbauer
coefficients. These estimates are optimal in the sense that
improvement in any negative power of $n$ is impossible for the $n$th
Gegenbauer coefficient $a_n^{\lambda}$.

\begin{theorem}\label{thm:decay rate gegenbauer coefficient}
Suppose that $f$ is analytic inside and on the Bernstein ellipse
$\mathcal{E}_{\rho}$ with $\rho
> 1$. When $- \frac{1}{2} < \lambda \leq 1$ and $\lambda \neq 0$, we
have
\begin{align}\label{eq:gengebauer coefficients bound I}
\left| a_n^{\lambda} \right| \leq \frac{ |c_{n,\lambda}| M
L(\mathcal{E}_{\rho} ) }{\pi \rho^{n+1}} {}_2\mathrm{
F}_1\left[\begin{matrix} n + 1, ~ 1 - \lambda;&
\\   n + \lambda + 1;  &\end{matrix} \hspace{-.25cm}  \frac{1}{ \rho^2 } \right],
\quad n \geq 0,
\end{align}
where $c_{n,\lambda}$ is defined by \eqref{eq:gegenbauer constant}
and $M = \max_{z\in \mathcal{E}_{\rho}} |f(z)|$ and
$L(\mathcal{E}_{\rho} )$ denotes the length of the circumference of
the ellipse $\mathcal{E}_{\rho}$. When $\lambda
> 1$, we have
\begin{align}\label{eq:gengebauer coefficients bound II}
\left| a_n^{\lambda} \right| \leq \frac{ c_{n,\lambda} M
L(\mathcal{E}_{\rho} ) }{\pi \rho^{n+1}} {}_2\mathrm{
F}_1\left[\begin{matrix} n + 1,~ 1 - \lambda;  &
\\  n + \lambda + 1; &\end{matrix} \hspace{-.4cm} -\frac{1}{ \rho^2 }
\right], \quad n\geq 0.
\end{align}
Finally, we point out that, apart from constant factors, the above
two bounds are optimal in the sense that these bounds on the right
hand side of \eqref{eq:gengebauer coefficients bound I} and
\eqref{eq:gengebauer coefficients bound II} can not be improved in
any negative power of $n$.
\end{theorem}

%\begin{theorem}\label{thm:decay rate gegenbauer coefficient}
%Suppose that $f$ is analytic
%inside and on the Bernstein ellipse $\mathcal{E}_{\rho}$ with $\rho
%> 1$. When $- \frac{1}{2} < \lambda \leq 1$ and $\lambda \neq 0$, we
%have
%\begin{align}\label{eq:gengebauer coefficients bound I}
%\left| a_n^{\lambda} \right| \leq \frac{ c_{n,\lambda} M
%L(\mathcal{E}_{\rho} ) }{\pi \rho^{n+1}} {}_2\mathrm{
%F}_1\left[\begin{matrix} n + 1, & 1 - \lambda;
%\\   n + \lambda + 1; \hspace{-1.2cm} &\end{matrix} \frac{1}{ \rho^2 } \right],
%\quad n \geq 0,
%\end{align}
%where  $M = \max_{z\in \mathcal{E}_{\rho}} |f(z)|$ and
%$L(\mathcal{E}_{\rho} )$ denotes the length of the circumference of
%the ellipse $\mathcal{E}_{\rho}$. When $\lambda > 1$, we have
%\begin{align}\label{eq:gengebauer coefficients bound II}
%\left| a_n^{\lambda} \right| \leq \frac{ c_{n,\lambda} M
%L(\mathcal{E}_{\rho} ) }{\pi \rho^{n+1}} \left(1 + \frac{1}{\rho^2}
%\right)^{\lambda-1}, \quad n \geq 0.
%\end{align}
%Finally, we point out that the above two bounds are optimal in the
%sense that these bounds on the right hand side of
%\eqref{eq:gengebauer coefficients bound I} and \eqref{eq:gengebauer
%coefficients bound II} can not be improved in any negative power of
%$n$.
%\end{theorem}
\begin{proof}
When $\lambda > -\frac{1}{2}$, from \eqref{eq:contour gegenbauer
coefficient}, it is clear that
\begin{align}\label{ineq:hypergeometric max bound}
\left| a_n^{\lambda} \right| \leq \frac{ |c_{n,\lambda}| M
L(\mathcal{E}_{\rho} ) }{\pi \rho^{n+1}} \max_{u\in
\mathcal{C}_{\rho}} \left| {}_2\mathrm{ F}_1\left[\begin{matrix} n +
1, & 1 - \lambda;
\\   n + \lambda + 1; \hspace{-1.2cm} &\end{matrix} \frac{1}{ u^2 } \right]
\right|,
\end{align}
where $\mathcal{C}_{\rho}$ denotes the circle $|u| = \rho$ and we
have used the fact that $u\in \mathcal{C}_{\rho}$ when we set $u =
z\pm \sqrt{z^2-1}$. We now turn to determine exactly where on the
circle the absolute value of the Gauss hypergeometric function takes
its maximum value. When $-\frac{1}{2} < \lambda \leq 1$ and $\lambda
\neq 0$, it is easy to see that the absolute value of the Gauss
hypergeometric function on the right hand side of above inequality
takes its maximum value at $u = \pm\rho$, this proves the inequality
\eqref{eq:gengebauer coefficients bound I}. For the case
$\lambda>1$, recall the Euler's integral representation of the Gauss
hypergeometric function (see \cite[p.~65]{andrews2000special}), we
have
\begin{align}\label{eq:hypergeometric function}
{}_2\mathrm{ F}_1\left[\begin{matrix} n + 1, & 1 - \lambda;
\\  n + \lambda + 1; \hspace{-1cm} &\end{matrix} \frac{1}{u^2}
\right] & = \frac{\Gamma(n+\lambda+1)}{\Gamma(n+1) \Gamma(\lambda)}
\int_{0}^{1} t^n (1 - t)^{\lambda-1} \left(1 -
\frac{t}{u^2}\right)^{\lambda-1} dt.
\end{align}
Hence
\begin{align}\label{ineq:hypergeometric bound}
\left| {}_2\mathrm{ F}_1\left[\begin{matrix} n + 1,~  1 - \lambda;&
\\  n + \lambda + 1;  &\end{matrix} \hspace{-.25cm} \frac{1}{u^2}
\right] \right| & \leq \frac{\Gamma(n+\lambda+1)}{\Gamma(n+1)
\Gamma(\lambda)} \int_{0}^{1} t^n (1 - t)^{\lambda-1} \left( 1 +
\frac{t}{\rho^2} \right)^{\lambda-1} dt \nonumber \\
& = {}_2\mathrm{ F}_1\left[\begin{matrix} n + 1,~  1 - \lambda; &
\\  n + \lambda + 1;  &\end{matrix} \hspace{-.4cm} -\frac{1}{\rho^2}
\right].
\end{align}
This implies that the absolute value of the Gauss hypergeometric
function takes its maximum value at $u = \pm i\rho$. This proves
\eqref{eq:gengebauer coefficients bound II}.

To prove the optimal property, we first prove the case
\eqref{eq:gengebauer coefficients bound I}. Consider the following
function
\begin{equation}\label{eq:real pole function}
f(x) = \frac{1}{x-b}, \quad b>1.
\end{equation}
Obviously, the function $f(x)$ has a simple pole at $z = b$ with
residue one. If we deform the contour of $a_n^{\lambda}$ as an
ellipse $\mathcal{E}_{R}$ in the positive direction and $R
> b + \sqrt{b^2 - 1}$, together with a sufficiently small circle with center at $z
= b$ in the negative direction and these two curves are connected by
a cross cut. Note that the integral along the cross cut vanishes, we
find that the Gegenbauer coefficient $a_n^{\lambda}$ is exactly the
difference between the contour integral on the right hand side of
\eqref{eq:contour gegenbauer coefficient} along $\mathcal{E}_{R}$
and the same integral along the small circle. Furthermore, note that
the contour integral along the ellipse $\mathcal{E}_{R}$ vanishes as
$R$ tends to infinity and the contour integral along the small
circle can be calculated explicitly by the Cauchy's integral
formula, we obtain
\begin{equation}\label{eq:pole function gegenbauer coefficients}
a_n^{\lambda} = - \frac{ 2c_{n,\lambda} }{ (b + \sqrt{b^2 -
1})^{n+1} } {}_2\mathrm{ F}_1\left[\begin{matrix} n + 1,~ 1-\lambda;
&
\\  n +\lambda + 1;  &\end{matrix} \hspace{-.25cm}  \frac{1}{ (b + \sqrt{b^2 - 1}
)^{2} } \right], \quad n \geq 0.
\end{equation}
Suppose that the bound on the right hand side of
\eqref{eq:gengebauer coefficients bound I} can be further improved
in some negative powers of $n$. Specifically, we suppose that the
Gegenbauer coefficient $a_n^{\lambda}$ further satisfies
\begin{align*}\label{ineq:pole coefficient}
\left| a_n^{\lambda} \right| \leq d_n \frac{ c_{n,\lambda} M
L(\mathcal{E}_{\rho} ) }{\pi \rho^{n+1}} {}_2\mathrm{
F}_1\left[\begin{matrix} n + 1,~ 1 - \lambda; &
\\  n + \lambda + 1;  &\end{matrix} \hspace{-.25cm}  \frac{1}{ \rho^2 } \right],
\end{align*}
where $d_n = \mathcal{O}\left( n^{-\delta} \right)$ for some
$\delta>0$. Setting $\rho = b + \sqrt{b^2 - 1} - \epsilon$ for any
$\epsilon>0$ and dividing both sides of the above inequality by the
bound of \eqref{eq:gengebauer coefficients bound I} yields
\[
\frac{2\pi}{ M L(\mathcal{E}_{\rho} ) } \left(\frac{b + \sqrt{b^2 -
1} - \epsilon}{b + \sqrt{b^2 - 1}} \right)^{n+1} \frac{ {}_2\mathrm{
F}_1\left[\begin{matrix} n + 1,~ 1-\lambda; &
\\  n +\lambda + 1;  &\end{matrix} \hspace{-.25cm} \frac{1}{ (b + \sqrt{b^2 - 1}
)^{2} } \right] }{ {}_2\mathrm{ F}_1\left[\begin{matrix} n + 1,~ 1 -
\lambda;&
\\   n + \lambda + 1;  &\end{matrix} \hspace{-.25cm}  \frac{1}{ (b + \sqrt{b^2 -
1} - \epsilon )^2 } \right] }  \leq d_n.
\]
Letting $n$ tend to infinity, note that both ${}_2\mathrm{ F}_1$ on
the left hand side of the above inequality tend to some bounded
constants (see Lemma \ref{lemma:asymptotic of hypergeometric}
below), we can always choose a sufficiently small $\epsilon$ such
that the left hand side is greater than the right hand side. This
results in an obvious contradiction. Thus, the bound
\eqref{eq:gengebauer coefficients bound I} is optimal in the sense
that improvement in any negative power of $n$ is impossible. This
proves the case \eqref{eq:gengebauer coefficients bound I}. The case
\eqref{eq:gengebauer coefficients bound II} can be handled in a
similar way.
%A slight difference is that, after dividing the right hand side of
%\eqref{eq:gengebauer coefficients bound II}, we obtain
%\[
%\frac{2\pi}{M L(\mathcal{E}_{\rho} ) } \frac{\left| {}_2\mathrm{
%F}_1\left[\begin{matrix} n + 1,~ 1-\lambda; &
%\\  n +\lambda + 1;  &\end{matrix} \hspace{-.25cm} \frac{1}{ (b + \sqrt{b^2 - 1}
%)^{2} } \right]  \right|}{ \left| {}_2\mathrm{
%F}_1\left[\begin{matrix} n + 1,~ 1 - \lambda;&
%\\   n + \lambda + 1;  &\end{matrix} \hspace{-.25cm}  \frac{-1}{ (b + \sqrt{b^2 - 1}
%)^2 } \right] \right| } \leq d_n.
%\]
%Note that both ${}_2\mathrm{ F}_1$ on the left hand side of the
%above inequality tend to some bounded constants as $n$ tends to
%infinity; see Lemma \ref{lemma:asymptotic of hypergeometric} below.
%Again, we derive a contradiction. Thus, for $\lambda > 1$, the bound
%\eqref{eq:gengebauer coefficients bound II} is optimal in the sense
%that improvement in any negative power of $n$ is impossible.
This completes the proof.
\end{proof}

The optimal estimates in Theorem \ref{thm:decay rate gegenbauer
coefficient} may be computationally expensive since they require the
computation of the Gauss hypergeometric function. Now we present
explicit estimates for the Gegenbauer coefficients and they are
achieved by finding explicit upper bounds for the Gauss
hypergeometric function and the constant $c_{n,\lambda}$.

The following inequality will be useful.
\begin{lemma}\label{lemma:inequality ratio gamma}
Let $n\geq1$ and $a,b\in \mathbb{R}$. For $n+a>1$ and $n+b>1$, we
have
\begin{align}
\frac{\Gamma(n+a)}{\Gamma(n+b)} \leq \Upsilon_{n}^{a,b} n^{a-b},
\end{align}
where
\begin{align}
\Upsilon_{n}^{a,b} = \exp\left( \frac{a-b}{2(n+b-1)} +
\frac{1}{12(n+a-1)} + \frac{(a-1)(a-b)}{n} \right).
\end{align}
\end{lemma}
\begin{proof}
See \cite[Lemma~2.1]{zhao2013sharp}.
\end{proof}

We now give explicit bounds for the Gegenbauer coefficients. These
bounds depend on the parameters $n$, $\rho$ and $\lambda$
explicitly.
\begin{theorem}\label{thm:explicit estimate}
Under the same assumptions of Theorem \ref{thm:decay rate gegenbauer
coefficient}. For $\lambda>0$ and any $n\geq1$, we have the
following explicit estimates
\begin{align}\label{eq:explicit estimate}
\left| a_n^{\lambda} \right| \leq \left\{\begin{array}{cc}
                                          {\displaystyle \Lambda(n,\rho,\lambda) \left(1 - \frac{1}{\rho^2}
\right)^{\lambda-1} \frac{n^{1-\lambda}}{ \rho^{n+1}} } , & \mbox{if
$0<\lambda\leq1$}, \\ [15pt]
                                          {\displaystyle \Lambda(n,\rho,\lambda) \left( 1 + \frac{1}{\rho^2}
\right)^{\lambda-1} \frac{n^{1-\lambda}}{ \rho^{n+1}} }, & \mbox{if
$\lambda>1$}.
                                        \end{array}
                                        \right.
\end{align}
where
\begin{align}
\Lambda(n,\rho,\lambda) = \frac{\Gamma(\lambda) M
\Upsilon_n^{1,\lambda}}{\pi} \left[ 2 \left(\rho + \frac{1}{\rho}
\right) + 2 \left( \frac{\pi}{2} - 1 \right) \left(\rho -
\frac{1}{\rho} \right) \right].
\end{align}
\end{theorem}
\begin{proof}
From \cite[Thm.~2.2.1]{andrews2000special}, it is clear that the
Euler's integral representation of the Gauss hypergeometric function
in \eqref{eq:hypergeometric function} is valid for all $\lambda>0$.
Therefore, for any $-1<x<1$
\begin{align}\label{eq:integral expression}
{}_2\mathrm{ F}_1\left[\begin{matrix} n + 1, ~ 1 - \lambda;&
\\   n + \lambda + 1;  &\end{matrix} \hspace{-.25cm}  x \right] & = \frac{\Gamma(n+\lambda+1)}{\Gamma(n+1) \Gamma(\lambda)}
\int_{0}^{1} t^n (1 - t)^{\lambda-1} (1-xt)^{\lambda-1} dt \nonumber
\\
& \leq \left\{\begin{array}{cc}
                                          {\displaystyle (1-|x|)^{\lambda-1} } , &
\mbox{if $0<\lambda\leq1$}, \\ [5pt]
                                          {\displaystyle (1+|x|)^{\lambda-1} }, &
\mbox{if $\lambda>1$}.
                                        \end{array}
                                        \right.
\end{align}
For the constant $c_{n,\lambda}$, by means of Lemma
\ref{lemma:inequality ratio gamma} we have $c_{n,\lambda}\leq
\Gamma(\lambda) \Upsilon_n^{1,\lambda} n^{1-\lambda}$. Moreover, the
perimeter of the ellipse $\mathcal{E}_{\rho}$ satisfies (see
\cite[Thm.~5]{jameson2014inequalities})
\begin{align}\label{ineq:bound ellipse}
L(\mathcal{E}_{\rho}) \leq 2 \left(\rho + \frac{1}{\rho} \right) + 2
\left( \frac{\pi}{2} - 1 \right) \left(\rho - \frac{1}{\rho}
\right), \quad \rho\geq1,
\end{align}
where the above inequality becomes an equality if $\rho=1$ or
$\rho\rightarrow\infty$. Combining these bounds with Theorem
\ref{thm:decay rate gegenbauer coefficient} gives us the desired
results.
\end{proof}

\begin{remark}
The restriction of $\lambda>0$ in Theorem \ref{thm:explicit
estimate} is due to the use of the Euler's integral representation
of the Gauss hypergeometric function. When $-\frac{1}{2}<\lambda<0$,
numerical experiments show that for any $-1<x<1$
\begin{align*}
{}_2\mathrm{ F}_1\left[\begin{matrix} n + 1, ~ 1 - \lambda;&
\\   n + \lambda + 1;  &\end{matrix} \hspace{-.25cm}  x \right] \leq
D_n(\lambda,\rho)(1-x)^{\lambda-1},
\end{align*}
where $D_n(\lambda,\rho)\approx1$ for large $n$.
\end{remark}

%\begin{remark}
%The perimeter of the ellipse $\mathcal{E}_{\rho}$
%%is given by
%%\cite[Eqn.~(19.9.9)]{olver2010nist}
%%\begin{align}
%%L(\mathcal{E}_{\rho}) = \frac{4}{\epsilon} E(\epsilon),
%%\end{align}
%%where $\epsilon=\frac{2}{\rho+\rho^{-1}}$ and $E(z)$ is the complete
%%ellipse integral of the second kind, which can be easily evaluated
%%by the composite trapezoidal rule; see
%%\cite{trefethen2014exponentially}. Meanwhile, it can be approximated
%%simply by (see, e.g., \cite[Eqn.~(19.9.10)]{olver2010nist})
%%\[
%%L(\mathcal{E}_{\rho}) \approx \pi\rho + \frac{3\pi}{10\rho^3 +
%%\rho\sqrt{4\rho^4-3}}.
%%\]
%%We also
%can be bounded by (see \cite[Thm.~5]{jameson2014inequalities})
%\begin{align}\label{ineq:bound ellipse}
%L(\mathcal{E}_{\rho}) \leq 2 \left(\rho + \frac{1}{\rho} \right) + 2
%\left( \frac{\pi}{2} - 1 \right) \left(\rho - \frac{1}{\rho}
%\right), \quad \rho\geq1,
%\end{align}
%where equality holds when $\rho=1$ or $\rho\rightarrow\infty$.
%\end{remark}

One of the most important cases of Gegenbauer expansion is the
Legendre expansion which is defined by
\begin{align}\label{eq:legendre expansion}
f(x) = \sum_{n= 0}^{\infty} a_n^L P_n(x), \quad  a_n^L = \left(n +
\frac{1}{2} \right) \int_{-1}^{1} f(x) P_n(x) dx.
\end{align}
In view of the first equality of \eqref{eq:gegenbauer and legendre},
so that $a_n^{L} = a_n^{\frac{1}{2}}$ for $n \geq 0$. As a direct
consequence of Theorem \ref{thm:explicit estimate}, we derive the
following explicit estimate for the Legendre coefficients.
\begin{corollary}
Under the same assumptions of Theorem \ref{thm:decay rate gegenbauer
coefficient}. For the Legendre coefficients $a_n^L$, we have
\begin{align}\label{eq:bound for legendre coefficient}
\left| a_n^{L} \right| \leq
\frac{\Lambda(n,\rho,\frac{1}{2})}{\sqrt{\rho^2-1}}
\frac{\sqrt{n}}{\rho^{n}} , \quad n \geq 1.
\end{align}
\end{corollary}
\begin{proof}
It follows immediately from \eqref{eq:explicit estimate} by setting
$\lambda = \frac{1}{2}$.
\end{proof}

In \cite[Thm.~2.7]{zhao2013sharp}, the authors derived the following
bound for the Legendre coefficients
\begin{equation}\label{eq:zhao's bound for legendre}
\left| a_n^{L} \right| \leq \frac{M\sqrt{\pi n}}{\rho^n} \left( 1 +
\frac{n+2}{2n+3} \frac{1}{\rho^2 - 1} \right) \exp\left(
\frac{8n-1}{12n(2n-1)} \right), \quad n\geq 1.
\end{equation}
%Now, we compare \eqref{eq:bound for legendre coefficient} and
%\eqref{eq:zhao's bound for legendre}. Note that our bound
%\eqref{eq:bound for legendre coefficient} involves the perimeter of
%the ellipse $\mathcal{E}_{\rho}$. We further replace it by the bound
%\eqref{ineq:bound ellipse} and this means that our bound depends on
%the parameters $\rho$ and $n$ explicitly. Numerical results are
%presented in the left part of Figure \ref{fig:comparison with zhao}.
%It is clear to see that our bound is sharper when $\rho$ is close to
%one. For large values of $\rho$, we can see that both bounds are
%almost the same.
In the left part of Figure \ref{fig:comparison with zhao}, we
compare these two bounds for several values of $\rho$. It is clear
to see that our bound is sharper when $\rho$ is close to one. For
large values of $\rho$, we can see that both bounds are almost the
same.

\begin{remark}
In \cite[Eqn.~(2.35)]{zhao2013sharp}, an estimate for the Gegenbauer
coefficients was given. Altering their result into our setting, it
can be written explicitly as
\begin{align}\label{eq:zhao's bound for gegenbauer}
|a_n^{\lambda}| & \leq A_n^{\lambda} M \left[
\frac{\sqrt{\pi}}{2^{2\lambda-1}} +
\frac{\Gamma(\lambda+\frac{1}{2})}{\sqrt{\Gamma(2\lambda+1)}}
\frac{2\sqrt{2}}{\rho^2-1} \right] \frac{\sqrt{n}}{\rho^n},
\end{align}
%\begin{align}\label{eq:zhao's bound for gegenbauer}
%|a_n^{\lambda}| & \leq \frac{A_n^{\lambda} M}{\rho^n} \left[
%\frac{\sqrt{\pi}}{2} \frac{ \Gamma(2n+2) \Gamma(n + 2\lambda) }{
%\Gamma(2n+2\lambda) \Gamma(n+\frac{3}{2}) } \nonumber \right.\\
%&~~~~~~~~~~~~~~\left. + \frac{2
%\Gamma(\lambda+\frac{1}{2})}{(\rho^2-1)
%\Gamma(n+\lambda+\frac{1}{2})} \sqrt{\frac{(2n+2\lambda) \Gamma(n+1)
%\Gamma(n+2\lambda)}{\Gamma(2\lambda+1)}} \right],
%\end{align}
where $M$ is defined as in Theorem \ref{thm:decay rate gegenbauer
coefficient} and
\[
A_n^{\lambda} = \frac{2^{4\lambda-2} \Gamma(\lambda+\frac{1}{2})
\Gamma(\lambda)^2}{\pi \Gamma(2\lambda)}
\frac{\Gamma(n+\lambda+\frac{1}{2})}{\Gamma(n+2\lambda)} \max\left\{
\Upsilon_{n}^{2\lambda,\frac{3}{2}} \Upsilon_{2n}^{2,2\lambda},
\sqrt{\frac{n+\lambda}{n}}
\sqrt{\Upsilon_{n}^{1,\lambda+\frac{1}{2}}
\Upsilon_{n}^{2\lambda,\lambda+\frac{1}{2}} } \right\}.
\]
In the right part of Figure \ref{fig:comparison with zhao}, we
present the ratio of our bound \eqref{eq:explicit estimate} to the
above bound for the case $\lambda= \frac{7}{2}$. We observe that our
bound is always superior, especially when $\rho$ is close to one.
Compared with the Legendre case, we also observe that our bound is
tighter than the above bound for large $\lambda$. In Figure
\ref{fig:comparison with zhao2} we further present a comparison of
our bound \eqref{eq:explicit estimate} to the above bound for two
larger values of $\lambda$. Clearly, we can see that the our bound
becomes increasingly sharper than the bound \eqref{eq:zhao's bound
for gegenbauer} with increasing $\lambda$.
\end{remark}

%\[
%A_n^{\lambda} = \frac{2^{4\lambda-2} \Gamma(\lambda+\frac{1}{2})
%\Gamma(\lambda)^2}{\pi \Gamma(2\lambda)}
%\frac{\Gamma(n+\lambda+\frac{1}{2})}{\Gamma(n+2\lambda)}.
%\]

\begin{figure}[ht]
\centering
\includegraphics[width=6.8cm]{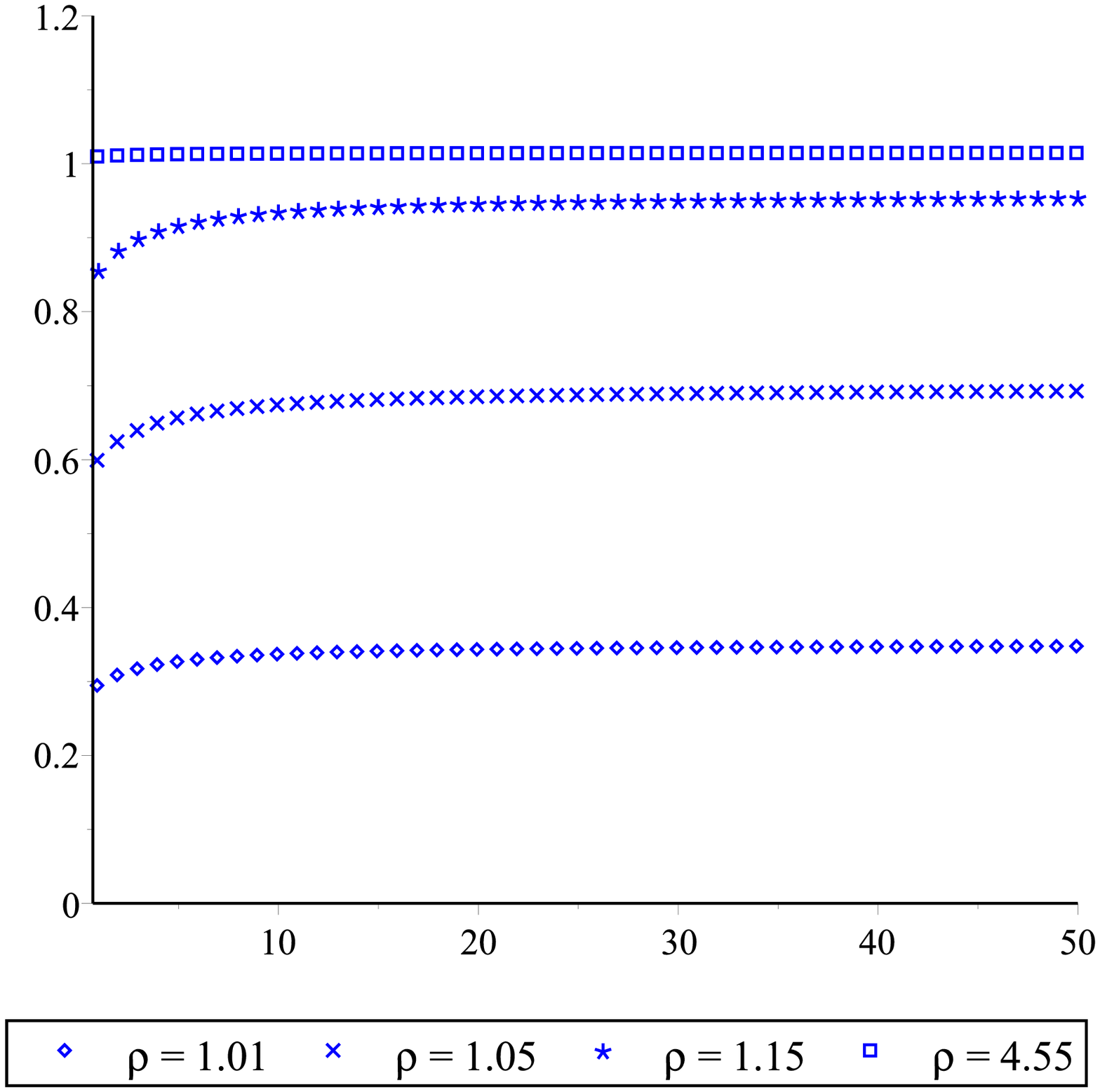}\qquad
\includegraphics[width=6.8cm]{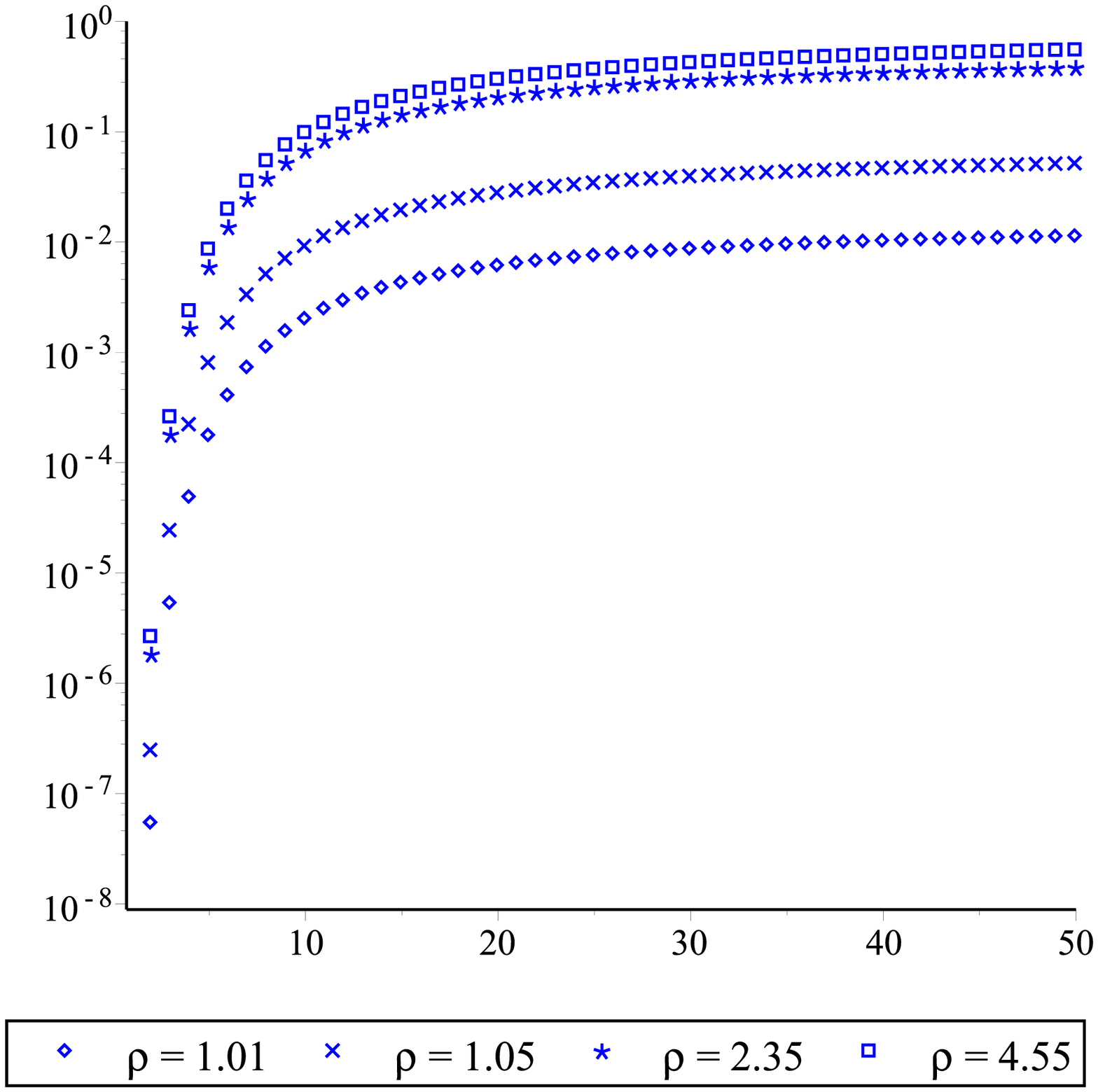}
\caption{The left panel shows the ratio of our bound \eqref{eq:bound
for legendre coefficient} to the bound \eqref{eq:zhao's bound for
legendre} for the Legendre coefficients. The right panel shows the
ratio of our bound \eqref{eq:explicit estimate} to the bound
\eqref{eq:zhao's bound for gegenbauer} for $\lambda = \frac{7}{2}$.
Here $n$ ranges from $1$ to $50$.} \label{fig:comparison with zhao}
\end{figure}

\begin{figure}[ht]
\centering
\includegraphics[width=6.8cm]{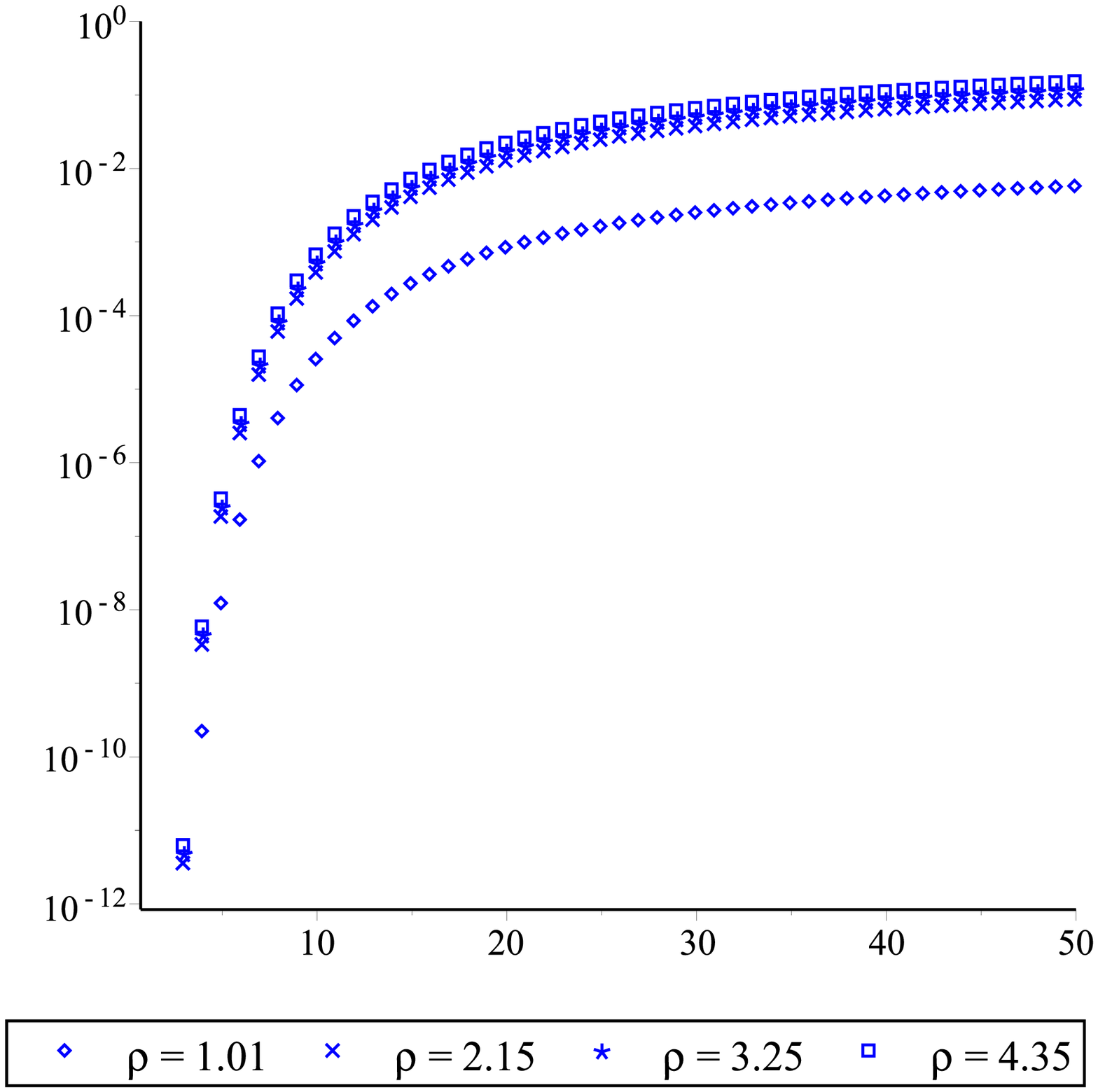}\qquad
\includegraphics[width=6.8cm]{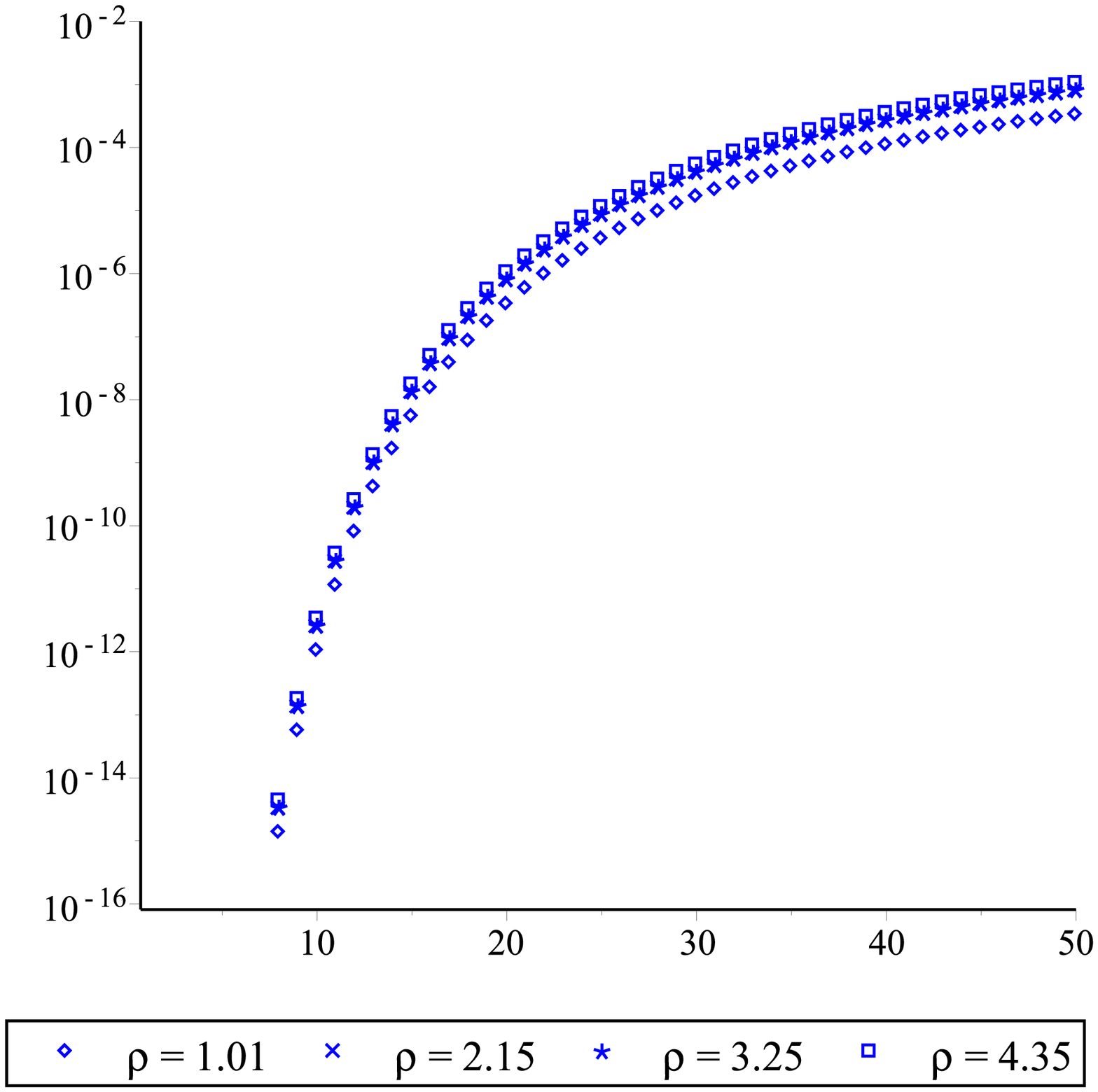}
\caption{The ratio of our bound \eqref{eq:explicit estimate} to the
bound \eqref{eq:zhao's bound for gegenbauer} for $\lambda =
\frac{11}{2}$ (left) and $\lambda=\frac{19}{2}$ (right). Here $n$
ranges from $1$ to $50$.} \label{fig:comparison with zhao2}
\end{figure}

\begin{remark}
For the function $Q_n^{(\lambda)}(z)$ which is defined in
\eqref{corollary:gegenbauer function second}, we use the proof of
Theorem \ref{thm:decay rate gegenbauer coefficient}, so that
\begin{align}\label{ineq:bound for gegenbauer transform}
\max_{z \in \mathcal{E}_{\rho}} |Q_n^{(\lambda)}(z)|
 = \left\{\begin{array}{ccc}
                                          {\displaystyle \frac{c_{n,\lambda} h_n^{(\lambda)}}{\rho^{n+1}} {}_2\mathrm{
F}_1\left[\begin{matrix} n + 1, & 1 - \lambda;
\\   n + \lambda + 1; \hspace{-1.2cm} &\end{matrix} \frac{1}{ \rho^2 } \right] },   & \mbox{$\textstyle -\frac{1}{2} < \lambda
\leq 1 $ and $\lambda \neq 0$},\\ [15pt]
                                          {\displaystyle \frac{c_{n,\lambda} h_n^{(\lambda)}}{\rho^{n+1}} {}_2\mathrm{
F}_1\left[\begin{matrix} n + 1, & 1 - \lambda;
\\   n + \lambda + 1; \hspace{-1.2cm} &\end{matrix} -\frac{1}{ \rho^2 } \right] },      & \mbox{$\lambda>1$.}
                                        \end{array}
                                        \right.
\end{align}
For the former case, the maximum value can be attained at $z = \pm
\frac{1}{2}(\rho + \rho^{-1})$. For the latter, the maximum value
can be attained at $z = \pm \frac{i}{2}(\rho - \rho^{-1})$. We
remark that \eqref{ineq:bound for gegenbauer transform} is very
useful in establishing some uniform and explicit bounds for
$Q_n^{(\lambda)}(z)$ or rigorous bounds for Gauss-Gegenbauer
quadrature (see, for example,
\cite{gautschi1983error,lederman2015analytical,zhao2013sharp}). For
example, when $\lambda = \frac{1}{2}$, $Q_n^{(\frac{1}{2})}(z)$ is
exactly the Legendre function of the second kind. As is well-known,
$Q_n^{(\frac{1}{2})}(x)$ is strictly monotonically decreasing on
$(1,\infty)$. Therefore, for $x>1+\delta$ and $\delta>0$, by setting
$\frac{1}{2}(\rho + \rho^{-1}) = 1 + \delta$ and using Lemma
\ref{lemma:inequality ratio gamma} and \eqref{eq:integral
expression}, we have
\begin{align}\label{eq:bound for legendre second}
|Q_n^{(\frac{1}{2})}(x) | < Q_n^{(\frac{1}{2})}(1+\delta) =
\sqrt{\pi} \frac{ \Gamma(n+1) }{ \Gamma(n+\frac{3}{2}) }
{}_2\mathrm{ F}_1\left[\begin{matrix} n + 1, & \frac{1}{2};
\\   n + \frac{3}{2}; \hspace{-0.25cm} &\end{matrix} \frac{1}{ \hat{\delta}^2 } \right] \frac{1}{
\hat{\delta}^{n+1}} \leq \frac{\sqrt{\pi}
\Upsilon_{n}^{1,\frac{3}{2}}}{\hat{\delta}^{n}
\sqrt{n(\hat{\delta}^2-1)}},
\end{align}
where $\hat{\delta} = 1 + \delta + \sqrt{(1+\delta)^2 - 1}$. In a
very recent paper of Lederman and Rokhlin
\cite[Lemma~3.1]{lederman2015analytical}, the authors provided the
following bound for $Q_n^{(\frac{1}{2})}(x)$,
\begin{equation}\label{eq:bound of rokhlin}
|Q_n^{(\frac{1}{2})}(x) | < \left( \log\left( 2
\frac{1+\tilde{\delta}}{\tilde{\delta}} \right) + 1 \right) \left(
\frac{1}{1+\tilde{\delta}} \right)^{n+1},
\end{equation}
where $\tilde{\delta} = \sqrt{(1 + \delta)^2 - 1}$. Figure
\ref{fig:comparison with rokhlin} illustrates the comparison of the
above two bounds. Clearly, we can see that our bound \eqref{eq:bound
for legendre second} is much sharper, especially when $n$ is large.
\end{remark}

\begin{figure}[ht]
\centering
\includegraphics[width=6.8cm]{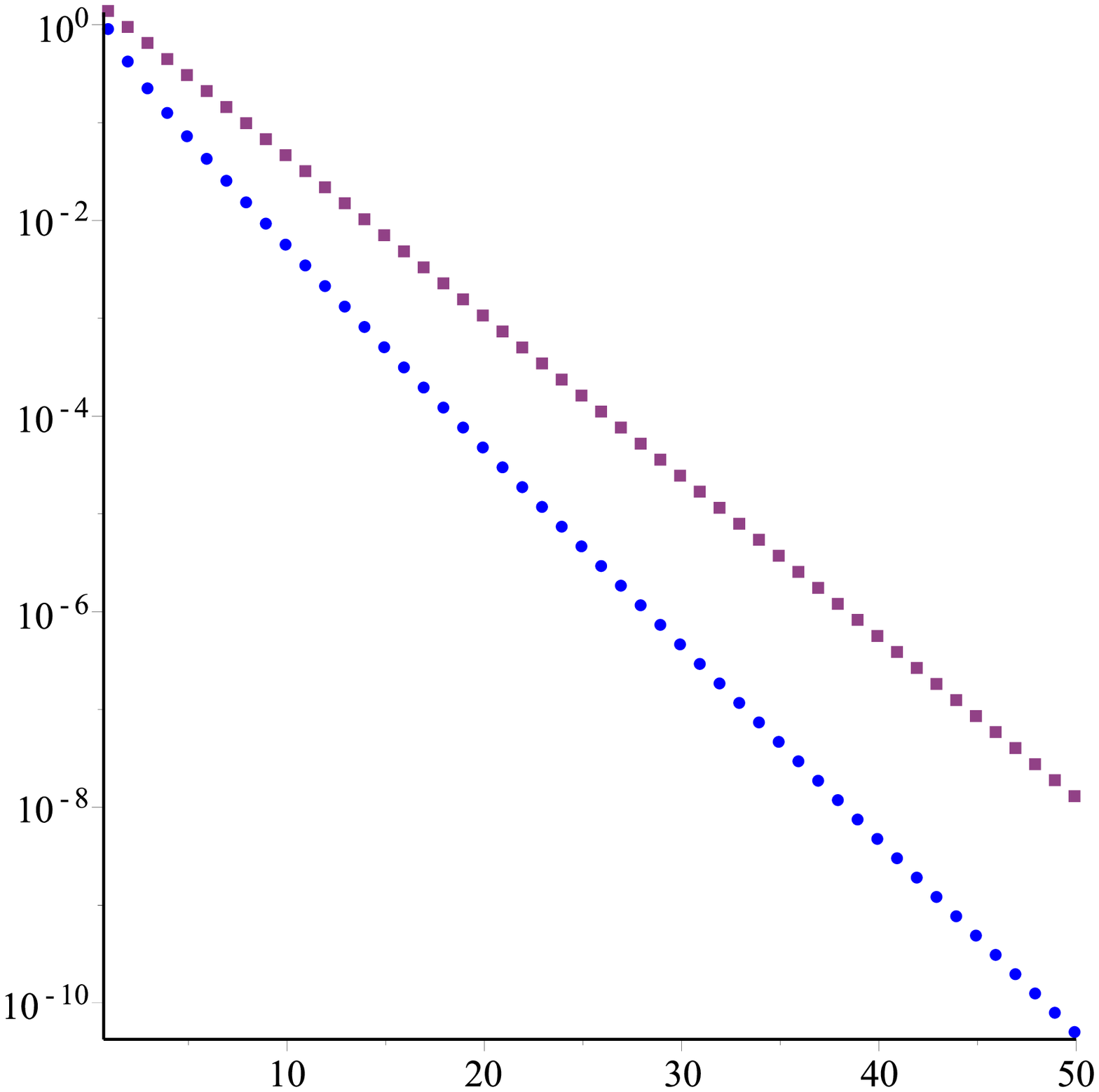}\qquad
\includegraphics[width=6.8cm]{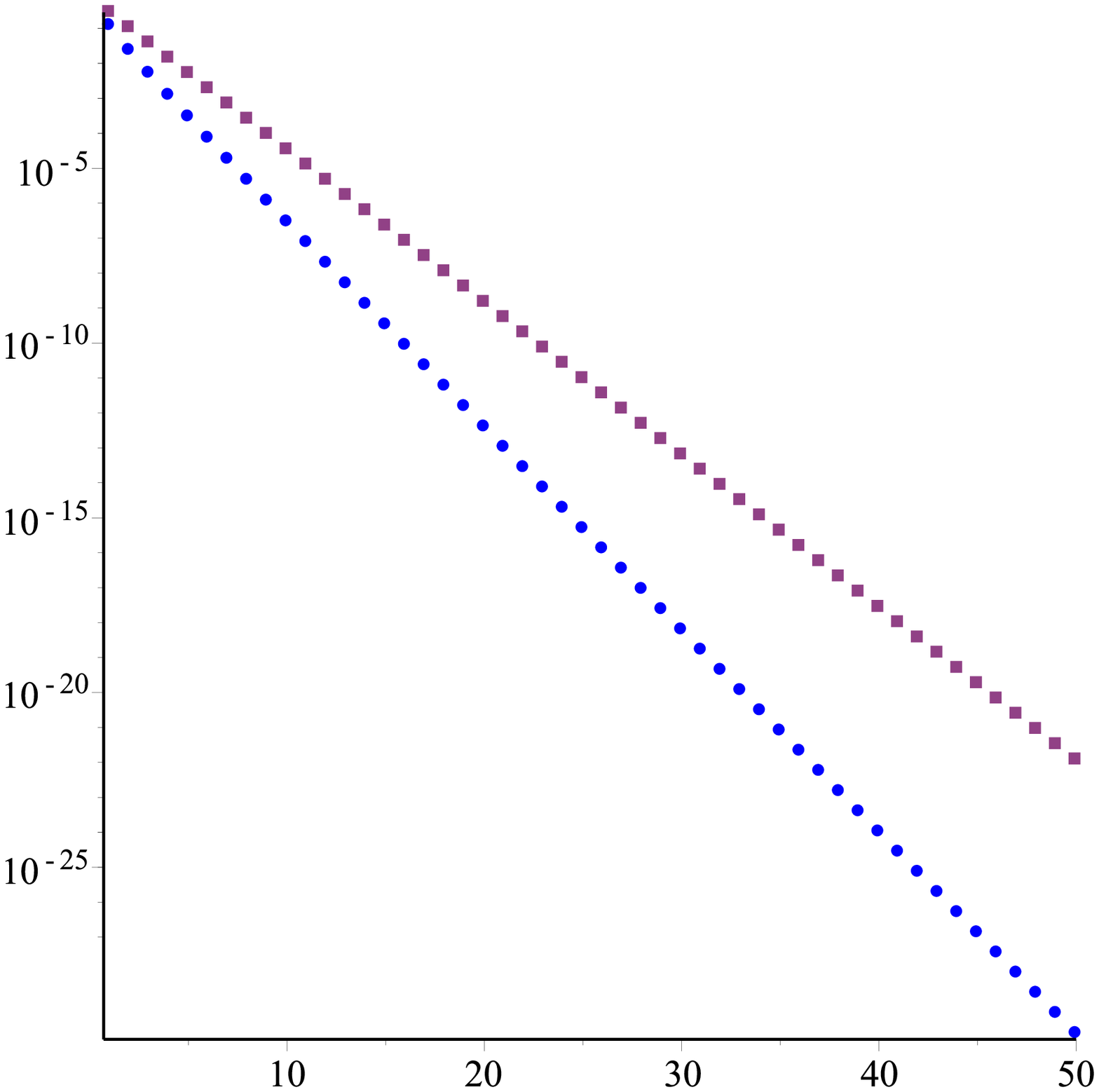}
\caption{Comparison of our bound \eqref{eq:bound for legendre
second} (dots) with the bound \eqref{eq:bound of rokhlin} (boxes)
for $\delta = 0.1$ (left) and $\delta = 1$ (right). Here $n$ ranges
from $1$ to $50$.} \label{fig:comparison with rokhlin}
\end{figure}

\subsection{Error bounds for truncated Gegenbauer expansion}
Having established bounds for the Gegenbauer coefficients, we can
immediately derive error bounds for the truncated Gegenbauer
expansion. Let
\begin{equation*}
f_N(x) = \sum_{n=0}^{N-1} a_n^{\lambda} C_n^{(\lambda)}(x).
\end{equation*}
We now give the first main result of this section.
\begin{theorem}\label{thm:error bound gegenbauer expansion}
Suppose that $f$ is analytic inside and on the Bernstein ellipse
$\mathcal{E}_{\rho}$ with $\rho
> 1$. When $0 < \lambda \leq 1$, the error of the truncated Gegenbauer expansion can be
bounded by
\begin{align}\label{ineq:bound gegenbauer series one}
\left| f(x) - f_N(x) \right| \leq  \frac{ M L(\mathcal{E}_{\rho})
\Gamma(\lambda) }{\pi \Gamma(2\lambda)} \sum_{n=N}^{\infty}
\frac{\Gamma(n+2\lambda)}{\Gamma(n+\lambda) \rho^{n+1}} {}_2\mathrm{
F}_1\left[\begin{matrix} n + 1,~ 1 - \lambda; &
\\  n + \lambda + 1;  &\end{matrix} \hspace{-.25cm} \frac{1}{ \rho^2 }
\right],
\end{align}
where $M$ and $L(\mathcal{E}_{\rho} )$ are defined as in Theorem
\ref{thm:decay rate gegenbauer coefficient}. In particular, when
$\lambda = 1$, the above bound can be given explicitly by
\begin{equation}\label{ineq:bound gegenbauer series two}
| f(x) - f_N(x) | \leq \frac{ M L(\mathcal{E}_{\rho}) }{\pi \rho^{N}
} \left( \frac{N(\rho - 1) + \rho }{ (\rho - 1)^2 } \right).
\end{equation}
When $\lambda
> 1$, the error of the truncated Gegenbauer series can be bounded by
\begin{align}\label{ineq:bound gegenbauer series three}
\left| f(x) - f_N(x) \right| \leq \frac{ M L(\mathcal{E}_{\rho})
\Gamma(\lambda) }{\pi \Gamma(2\lambda)} \sum_{n=N}^{\infty}
\frac{\Gamma(n+2\lambda)}{\Gamma(n+\lambda) \rho^{n+1}} {}_2\mathrm{
F}_1\left[\begin{matrix} n + 1, ~ 1 - \lambda; &
\\  n + \lambda + 1;  &\end{matrix} \hspace{-.4cm} -\frac{1}{ \rho^2 }
\right].
\end{align}
\end{theorem}
\begin{proof}
Using \eqref{eq:gegenbauer inequality} we have
\begin{align}\label{eq:bound of gegenbauer series}
\left| f(x) - f_N(x) \right| &= \left| \sum_{n = N}^{\infty}
a_n^{\lambda} C_n^{(\lambda)}(x) \right| \leq \sum_{n = N}^{\infty}
\left| a_n^{\lambda} \right| C_n^{(\lambda)}(1).
\end{align}
When $0 < \lambda \leq 1$, by applying \eqref{eq:normalization
gegenbauer} and \eqref{eq:gengebauer coefficients bound I}, one
finds
\begin{align*}
\left| f(x) - f_N(x) \right| & \leq \sum_{n = N}^{\infty} \frac{
c_{n,\lambda} M L(\mathcal{E}_{\rho} ) }{\pi \rho^{n+1}}
{}_2\mathrm{ F}_1\left[\begin{matrix} n + 1,~  1 - \lambda; &
\\  n + \lambda + 1;  &\end{matrix} \hspace{-.25cm} \frac{1}{ \rho^2 } \right]
C_n^{(\lambda)}(1) \nonumber \\
& = \frac{ M L(\mathcal{E}_{\rho}) \Gamma(\lambda) }{\pi
\Gamma(2\lambda)} \sum_{n=N}^{\infty}
\frac{\Gamma(n+2\lambda)}{\Gamma(n+\lambda) \rho^{n+1}} {}_2\mathrm{
F}_1\left[\begin{matrix} n + 1,~ 1 - \lambda; &
\\   n + \lambda + 1;  &\end{matrix} \hspace{-.25cm} \frac{1}{ \rho^2 }
\right].
\end{align*}
This proves \eqref{ineq:bound gegenbauer series one}. When $\lambda
= 1$, note that the Gauss hypergeometric function on the right hand
side of the last equation is reduced to one and the above inequality
can be further simplified as
\begin{align*}
\left| f(x) - f_N(x) \right| & \leq \frac{ M L(\mathcal{E}_{\rho})
}{\pi } \sum_{n=N}^{\infty} \frac{n+1}{\rho^{n+1}} \nonumber \\
&=\frac{ M L(\mathcal{E}_{\rho}) }{\pi \rho^{N} } \left(
\frac{N(\rho - 1) + \rho }{ (\rho - 1)^2 } \right).
\end{align*}
This proves \eqref{ineq:bound gegenbauer series two}. Similarly, the
case $\lambda > 1$ can be proved with the use of
\eqref{eq:gengebauer coefficients bound II} and \eqref{eq:bound of
gegenbauer series}. This completes the proof.

\end{proof}

%To reconstruct a nonperiodic function from its truncated Fourier
%series, a Gegenbauer reconstruction method was developed by
%reexpanding the Fourier series into a truncated Gegenbauer series
%(see \cite{gottlieb1992gibbs}). To achieve high resolution recovery
%of the function, the Gegenbauer reconstruction method requires that
%both the order of the Gegenbauer polynomials and the number of terms
%of the truncated Gegenbauer series grow linearly with respect to the
%degree of the Fourier series. Boyd in \cite{boyd2005trouble} studied
%this

The bound for the case $\lambda>1$ is very useful for analyzing the
convergence of the diagonal Gegenbauer approximation (see, e.g.,
\cite{boyd2005trouble}). In what follows, we consider to establish a
new upper bound which is simpler but surely less sharp than
\eqref{ineq:bound gegenbauer series three}.

\begin{theorem}\label{thm:gegenbauer simple bound}
Under the same assumptions of Theorem \ref{thm:error bound
gegenbauer expansion}. If $\lambda>1$ and $\rho >
\frac{(N+2\lambda)(N+\lambda+1)}{(N+\lambda)(N+1)}$, then
\begin{equation}\label{ineq:bound gegenbauer series four}
\left| f(x) - f_N(x) \right| \leq \frac{ \Gamma(\lambda)
\Gamma(N+2\lambda) }{ \Gamma(2\lambda) \Gamma(N+\lambda) }
\left(1+\frac{1}{\rho^2} \right)^{\lambda-1} \frac{C}{\rho^N} .
\end{equation}
where $C$ is defined by
\[
C = \frac{M L(\mathcal{E}_{\rho}) }{\pi}
 \left( \frac{
(N+\lambda)(N+1) }{ \rho(N+\lambda)(N+1) - (N+2\lambda)(N+\lambda+1)
} \right).
\]
\end{theorem}
\begin{proof}
Define
\begin{align}\label{def:auxiliary function}
H(n) = \frac{\Gamma(n+2\lambda)}{\Gamma(n+\lambda) \rho^{n+1}}
{}_2\mathrm{ F}_1\left[\begin{matrix} n + 1, ~ 1 - \lambda; &
\\  n + \lambda + 1;  &\end{matrix} \hspace{-.4cm} -\frac{1}{ \rho^2 }
\right].
\end{align}
Then, we have for $n\geq N$ that
\[
\frac{H(n+1)}{H(n)} = \frac{n+2\lambda}{\rho( n+\lambda)} \frac{
{}_2\mathrm{ F}_1\left[\begin{matrix} n + 2, ~ 1 - \lambda; &
\\  n + \lambda + 2;  &\end{matrix} \hspace{-.4cm} -\frac{1}{ \rho^2 }
\right] }{ {}_2\mathrm{ F}_1\left[\begin{matrix} n + 1, ~ 1 -
\lambda; &
\\  n + \lambda + 1;  &\end{matrix} \hspace{-.4cm} -\frac{1}{ \rho^2 }
\right] }.
\]
Applying \eqref{eq:hypergeometric function} to the ratio of Gauss
hypergeometric functions yields
\begin{align*}
\frac{ {}_2\mathrm{ F}_1\left[\begin{matrix} n + 2, ~ 1 - \lambda; &
\\  n + \lambda + 2;  &\end{matrix} \hspace{-.4cm} -\frac{1}{ \rho^2 }
\right] }{ {}_2\mathrm{ F}_1\left[\begin{matrix} n + 1, ~ 1 -
\lambda; &
\\  n + \lambda + 1;  &\end{matrix} \hspace{-.4cm} -\frac{1}{ \rho^2 }
\right] } & = \frac{n+\lambda+1}{n+1} \frac{ \int_{0}^{1}
t^{n+1}(1-t)^{1-\lambda} (1+\frac{t}{\rho^2})^{\lambda-1}dt }{
\int_{0}^{1} t^{n}(1-t)^{1-\lambda}
(1+\frac{t}{\rho^2})^{\lambda-1}dt } \leq \frac{n+\lambda+1}{n+1}.
\end{align*}
Therefore,
\[
\frac{H(n+1)}{H(n)} \leq \frac{( n+2\lambda) (n+\lambda+1) }{ \rho
(n+\lambda)(n+1) } \leq \frac{( N+2\lambda) (N+\lambda+1) }{ \rho(
N+\lambda)(N+1) }.
\]
Owing to the assumption on $\rho$, we see that the last bound is
less than one. Combining this with \eqref{ineq:bound gegenbauer
series three} gives
\begin{align*}
|f(x) - f_N(x)| & \leq \frac{ M L(\mathcal{E}_{\rho})
\Gamma(\lambda) }{\pi \Gamma(2\lambda)} \sum_{n=N}^{\infty} H(n) \nonumber \\
& \leq \frac{ M L(\mathcal{E}_{\rho}) \Gamma(\lambda) }{\pi
\Gamma(2\lambda)} H(N) \sum_{m=0}^{\infty} \left( \frac{(
N+2\lambda) (N+\lambda+1) }{ \rho(
N+\lambda)(N+1) } \right)^m \\
&= \frac{ M L(\mathcal{E}_{\rho}) \Gamma(\lambda) }{\pi
\Gamma(2\lambda)} \frac{ H(N) \rho
(N+\lambda)(N+1)}{\rho(N+\lambda)(N+1)-(N+2\lambda)(N+\lambda+1)}.
\nonumber
\end{align*}
Combining this with \eqref{eq:integral expression} gives the desired
result. This completes the proof.
\end{proof}

To overcome the Gibbs phenomenon, a Gegenbauer reconstruction method
was developed in
\cite{gottlieb1995gibbs1,gottlieb1995gibbs2,gottlieb1996gibbs,gottlieb1997gibbs,gottlieb1992gibbs}
and this method required to analyze the convergence of the diagonal
Gegenbauer approximation, i.e., $\lambda$ depends linearly on $N$.
We also refer the reader to \cite{boyd2005trouble} for further
discussion. In this case, our bound \eqref{ineq:bound gegenbauer
series four} provides a rigorous and computable bound.
\begin{corollary}\label{corollary:bound for Shu's method}
If $\lambda = \gamma N>1$ and $\gamma$ is a positive constant and
$\rho>\frac{(1+2\gamma)( (1+\gamma) N + 1)}{(1+\gamma)(N+1)}$, then
the error of the truncated Gegenbauer series can be bounded by
\begin{equation}\label{ineq:bound for Shu's method}
|f(x) - f_N(x)| \leq \frac{ \Gamma(\gamma N) \Gamma((1+2\gamma)N)}{
\Gamma(2\gamma N) \Gamma((1+\gamma)N)} \left(1+\frac{1}{\rho^2}
\right)^{\gamma N-1} \frac{\tilde{C}}{\rho^N},
\end{equation}
where
\[
\tilde{C} = \frac{ M L(\mathcal{E}_{\rho}) }{\pi} \left( \frac{
(1+\gamma)(N+1) }{ \rho(1+\gamma)(N+1) - (1+2\gamma)( (1+\gamma)N +
1 ) } \right).
\]
\end{corollary}
\begin{proof}
If follows directly from Theorem \ref{thm:gegenbauer simple bound}
by setting $\lambda = \gamma N$.
\end{proof}

\begin{example}
We illustrate the error bound \eqref{ineq:bound for Shu's method}
for the function $f(x) = \frac{1}{x-2}$. It is easy to check that
this function is analytic inside the ellipse $\mathcal{E}_{\rho}$
with $1<\rho<2+\sqrt{3}\approx 3.732$ and
\[
\max_{z\in \mathcal{E}_{\rho}}|f(z)| \leq M =  \frac{2\rho}{(
2+\sqrt{3}-\rho )( \rho-(2-\sqrt{3}) )}.
\]
We compare our bound \eqref{ineq:bound for Shu's method} with the
maximum error of the truncated Gegenbauer series
\[
\max_{x\in[-1,1]}|f(x) - f_N(x)|
\] which is measured at $1000$
equispaced points in $[-1,1]$. In our computations, we choose $\rho
= 3.6$ and the perimeter of the ellipse is evaluated by the
following formula \cite[Eqn.~(19.9.9)]{olver2010nist}
\begin{equation}\label{eq:length of ellipse}
L(\mathcal{E}_{\rho} ) = \frac{4 }{\epsilon} E(\epsilon),
\end{equation}
where $\epsilon = \frac{2}{\rho + \rho^{-1}}$ and $E(z)$ is the
complete ellipse integral of the second kind, although we remark
that simple approximation formulas are also available (see, for
example, \cite[Eqn.~(19.9.10)]{olver2010nist}). We test two values
$\gamma = 0.25$ and $\gamma=0.125$ and it is easy to check that
these values of $\rho$ and $\gamma$ satisfy the assumptions of
Corollary \ref{corollary:bound for Shu's method} when $N \geq 8$.
All computations were performed using {\sc Maple} with $100$ digits
arithmetic. Numerical results are presented in Figure
\ref{fig:comparison with exact maximum error}, which show that our
bound is tight.
\end{example}

\begin{figure}[ht]
\centering
\includegraphics[width=6.8cm]{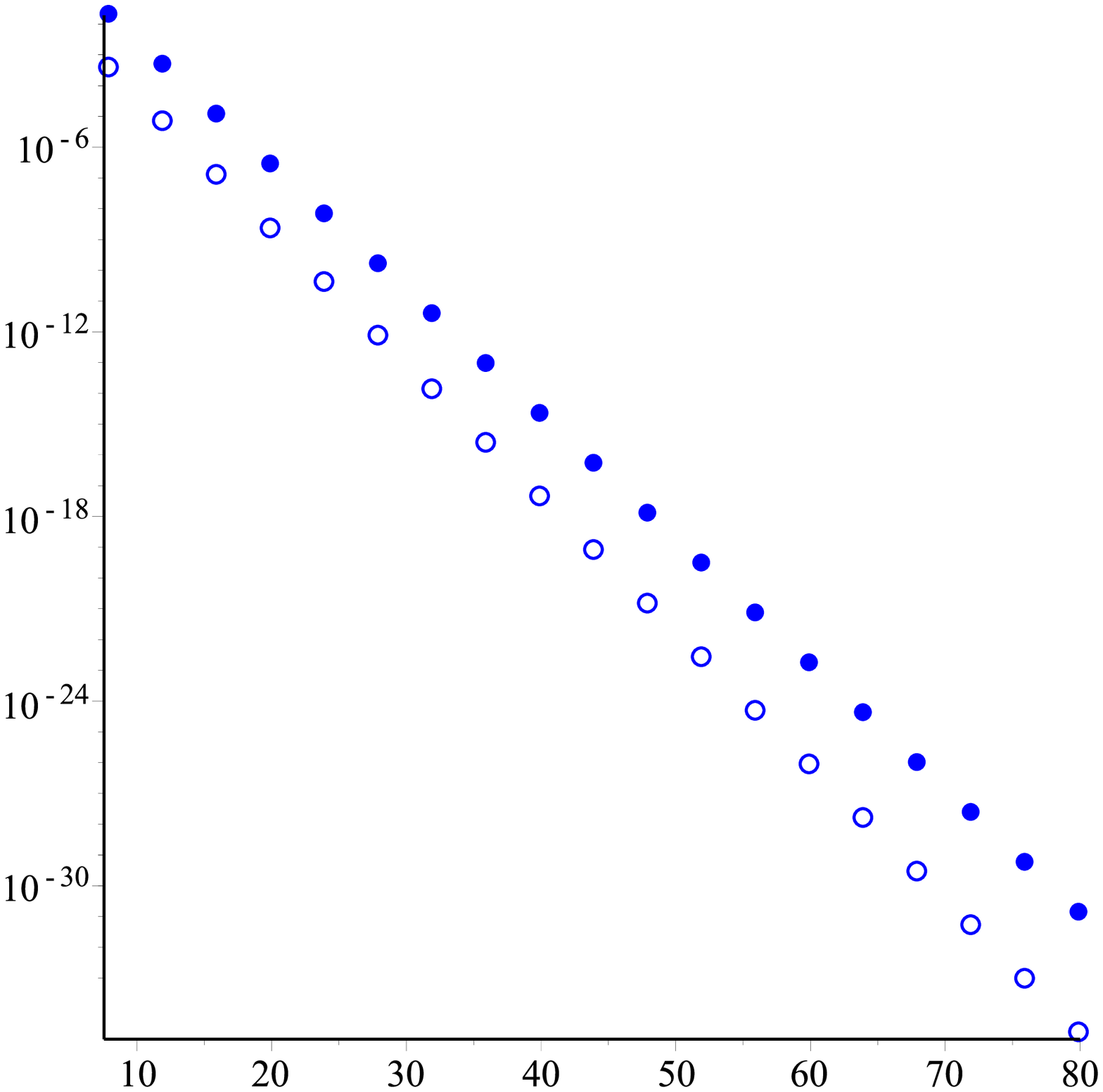}\qquad
\includegraphics[width=6.8cm]{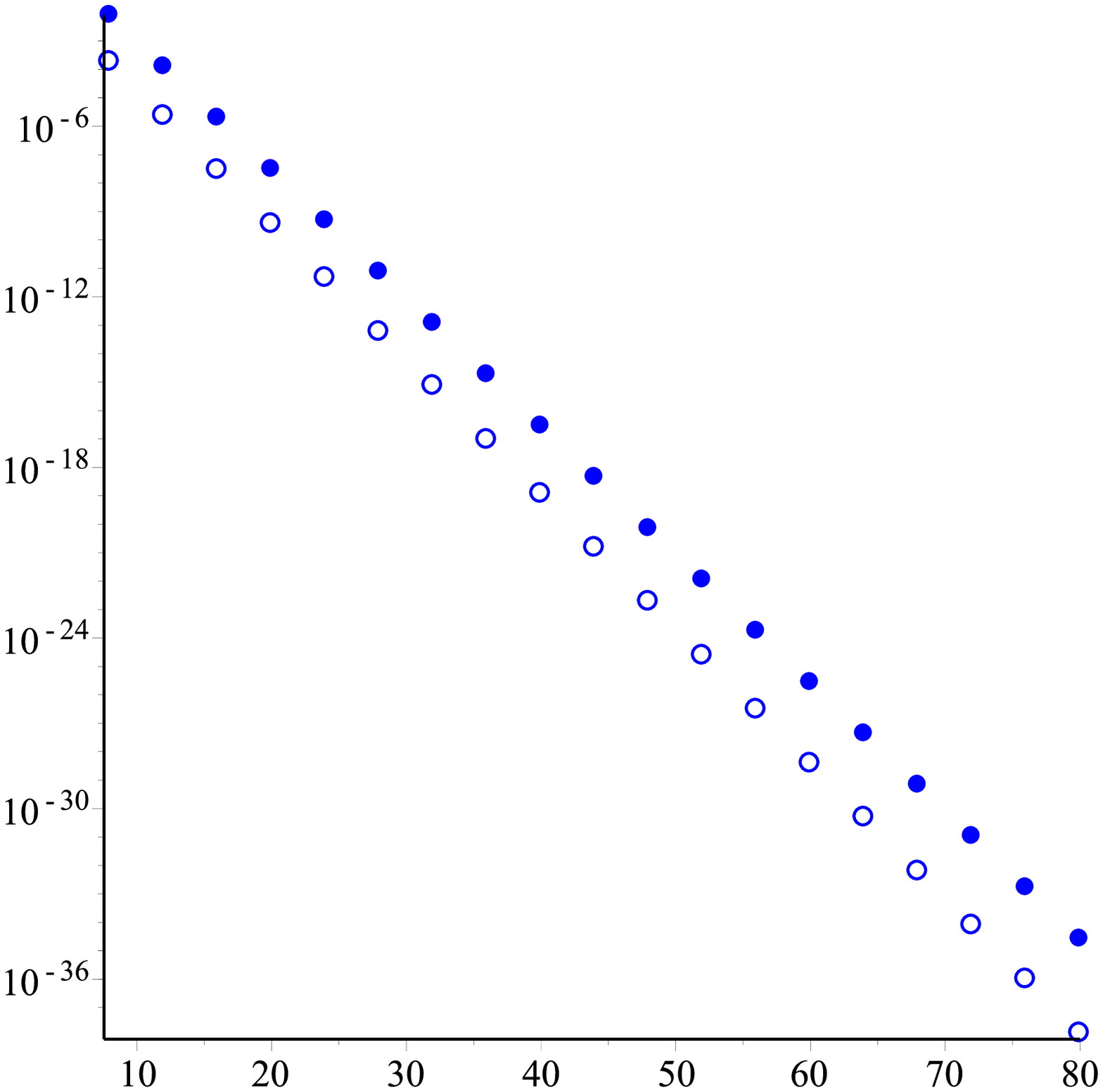}
\caption{Comparison of our bound \eqref{ineq:bound for Shu's method}
(dots) with the maximum error (circles) for $\gamma = 0.25$ (left)
and $\gamma = 0.125$ (right). Here $N$ ranges from $8$ to $80$.}
\label{fig:comparison with exact maximum error}
\end{figure}

\begin{remark}
It was demonstrated in \cite{boyd2005trouble} that the singularities
of $f(x)$ in the complex plane can ruin the convergence of the
diagonal Gegenbauer approximation method unless the constant
$\gamma$ is sufficiently small. This means that, to ensure the
convergence of the diagonal Gegenbauer approximation method, there
should be a lower bound on $\rho$ for each $\gamma$. The condition
of $\rho$ in Corollary \ref{corollary:bound for Shu's method}
provides a lower bound.
\end{remark}

\begin{remark}
For each fixed $\gamma$, the error bound \eqref{ineq:bound for Shu's
method} can be further optimized as a function of $\rho$.
\end{remark}

%\begin{remark}
%To analyze the asymptotic behaviour of the bound, it is necessary to
%derive the asymptotic behaviour of the following function
%\[
%{}_2\mathrm{ F}_1\left[\begin{matrix} N + 1, ~ 1 - \gamma N; &
%\\  (\gamma+1)N + 1;  &\end{matrix} \hspace{-.4cm} -\frac{1}{ \rho^2 }
%\right],
%\]
%as $N\rightarrow\infty$. However, still do not know whether such a
%result exists in the literature.
%\end{remark}

\section{A comparison of Legendre and Chebyshev
coefficients}\label{sec:comparison} The most commonly used cases of
Gegenbauer expansion are the Chebyshev and Legendre expansions
which, up to normalization, correspond to the special cases $\lambda
= 0$ and $\lambda = \frac{1}{2}$, respectively. More specifically,
the Legendre expansion is defined in \eqref{eq:legendre expansion}
and the Chebyshev expansion of the first kind is defined by
\begin{align}\label{eq:chebyshev expansion first}
f(x) = \sum_{n = 0}^{\infty}{'} a_n^C T_n(x), \quad  a_n^C =
\frac{2}{\pi} \int_{-1}^{1} \frac{f(x) T_n(x) }{\sqrt{1 - x^2}} dx,
\end{align}
where the prime indicates that the first term of the sum is halved.

For both expansions, one particularly interesting question is the
comparison of the decay rates of the Chebyshev and Legendre
coefficients (see
\cite{boyd2013relationships,fox1968chebyshev,lanczos1952introduction,wang2012convergence}).
For few decades, a myth on this issue is the ``Lanczos-Fox-Parker"
proposition which states that the Chebyshev coefficient $a_n^C$
decays approximately $\sqrt{n\pi}/2$ faster than the Legendre
coefficient $a_n^L$ for large $n$. This proposition was first
advocated by Lanczos in \cite{lanczos1952introduction} and later by
Fox and Parker in their monograph \cite[p.~17]{fox1968chebyshev}.
The original deviation of this proposition was based on the use of
Rodrigues' formula of orthogonal polynomials and repeated
integration by parts, which is not rigorous enough (see
\cite[p.~130]{mason2003chebyshev}). Recently, Boyd and Petschek in
\cite{boyd2013relationships} considered this issue and showed that
the ``Lanczos-Fox-Parker" proposition is not true for the following
concrete examples
\begin{equation*}\label{eq:three excepts}
f(x) = \left\{\begin{array}{ccc}
                                          {\displaystyle \exp(r T_m(x))},   & \mbox{$\textstyle  r>0$ and
                                          $m$ is a large integer},\\ [6pt]
                                          {\displaystyle \frac{1}{r - x}},   & \mbox{$r > 1$},
                                          \\ [12pt]
                                          {\displaystyle (1-x)^{\phi}},      & \mbox{$\phi>0$ is not an integer.}
                                        \end{array}
                                        \right.
\end{equation*}
Even these counterexamples were given, however, a precise result on
the comparison of Legendre and Chebyshev coefficients is still
nontrivial. In this section, inspired by these exceptions, we
analyze this issue based on the contour integral expression of
Gegenbauer coefficients. To this aim, we define the ratio of the
$n$th Legendre coefficient to the $n$th Chebyshev coefficient as
\begin{equation}
\gamma_n = \frac{ a_n^L }{ a_n^C }, \quad  a_n^C \neq 0, \quad n
\geq 0.
\end{equation}
We shall discuss the asymptotic behaviour of $\gamma_n$.

Now we give the relation between the Gegenbauer coefficients and the
Chebyshev and Legendre coefficients which will be used in our
analysis.
\begin{lemma}\label{lemma:legendre and chebyshev}
Suppose that $a_n^L$, $a_n^C$ and $a_n^{\lambda}$ denote the
Legendre, Chebyshev and Gegenbauer coefficients of the the same
function $f(x)$, which are defined by \eqref{eq:legendre expansion},
\eqref{eq:chebyshev expansion first}, and \eqref{eq:gegenbauer
coefficients}, respectively. Then $a_n^L = a_n^{\frac{1}{2}}$ for
each $n \geq 0$, and
\begin{align*}\label{eq:chebyshev and gegenbauer}
a_n^C = \frac{2}{n} \lim_{\lambda\rightarrow0^{+}} \lambda
a_n^{\lambda}, \quad n \geq 1,
\end{align*}
and $a_0^C = 2 \lim_{\lambda\rightarrow0^{+}} a_0^{\lambda}$.
\end{lemma}
\begin{proof}
The Legendre case follows from \eqref{eq:gegenbauer and legendre}.
As for the Chebyshev case, it follows directly from
\eqref{eq:gegenbauer and chebyshev first} for $n \geq 1$ and the
fact that $C_0^{(0)}(x) = 1$ for $n = 0$.
\end{proof}

For simplicity of presentation, we assume that $f(z)$ has a single
singularity such as a pole or branch point at $z = z_0$ in the
complex plane, although it is not difficult to extend our results to
the case that $f(z)$ has finite numbers of singularities. We
distinguish three cases $z_0 \in \mathbb{C}\setminus [-1,1]$, $z_0 =
\pm 1$ and $z_0 \in (-1,1)$.

\subsection{The case $z_0 \in \mathbb{C}\setminus [-1,1]$}
Before embarking on our analysis, we first introduce some helpful
lemmas.
%which we will use in the rest.
\begin{lemma}\label{lemma:asymptotic legendre constant}
We have
\begin{equation}
c_{n,\frac{1}{2}} = \sqrt{n \pi} \left( 1 + \frac{1}{8n} +
\mathcal{O}(n^{-2}) \right), \quad n\rightarrow\infty.
\end{equation}
\end{lemma}
\begin{proof}
This is a direct consequence of Remark \ref{remark:asymptotic of c}.
\end{proof}

\begin{lemma}\label{lemma:asymptotic of hypergeometric}
For fixed $x \neq 1$, we have
\begin{align}
{}_2\mathrm{ F}_1\left[\begin{matrix} a + \lambda,~  b; &
\\  c + \lambda; &\end{matrix} \hspace{-.25cm} x \right] = (1 - x)^{-b} \left[ 1 + \frac{b(c-a)}{c+\lambda} \left( \frac{x}{1-x} \right) + \mathcal{O}(\lambda^{-2})
\right],
\end{align}
as $ \lambda\rightarrow\infty$.
\end{lemma}
\begin{proof}
See \cite{temme2003large}.
\end{proof}

Consider the following model function
\begin{align}\label{def:pole function}
f(x) = \frac{1}{x - z_0}, \quad z_0 \in \mathbb{C}\setminus [-1,1],
\end{align}
which has a simple pole at $x = z_0$. This function provides
valuable insights for the asymptotic behaviour of $\gamma_n$ for
large $n$. Our main result is stated in the following theorem.
\begin{theorem}\label{thm:asymptotic ratio for pole function}
For the function \eqref{def:pole function}, we have
\begin{align}\label{eq:asymptotic of ratio}
\gamma_n = g(z_0) \sqrt{n\pi} + \mathcal{O}(n^{-\frac{1}{2}}), \quad
n\rightarrow\infty,
\end{align}
where $g(z)$ is defined by
\begin{align}\label{def:function g}
g(z) = \sqrt{\frac{z^2 - 1}{ (z \pm \sqrt{z^2 - 1})^2 - 1 }},
\end{align}
and the sign is chosen such that $|z \pm \sqrt{z^2 - 1}|>1$ and
$g(z)>0$ if $z>1$.
\end{theorem}
\begin{proof}
If follows from \eqref{eq:real pole function} with $b = z_0$ that
\begin{equation*}
a_n^{\lambda} = - \frac{ 2c_{n,\lambda} }{ (z_0 \pm \sqrt{z_0^2 -
1})^{n+1} } {}_2\mathrm{ F}_1\left[\begin{matrix} n + 1,~ 1-\lambda;
& \\  n +\lambda + 1;  &\end{matrix} \hspace{-.25cm} \frac{1}{ (z_0
\pm \sqrt{z_0^2 - 1} )^{2} } \right], \quad n \geq 0.
\end{equation*}
By the Lemma \ref{lemma:legendre and chebyshev},
\begin{equation*}\label{eq:legendre coefficients for pole function}
a_n^L = - \frac{ 2c_{n,\frac{1}{2}} }{ (z_0 \pm \sqrt{z_0^2 -
1})^{n+1} } {}_2\mathrm{ F}_1\left[\begin{matrix} n + 1,~
\frac{1}{2}; &
\\  n + \frac{3}{2};  &\end{matrix} \hspace{-0.25cm} \frac{1}{ (z_0 \pm \sqrt{z_0^2 - 1}
)^{2} } \right], \quad n \geq 0,
\end{equation*}
and
\begin{equation*}\label{eq:chebyshev coefficients for pole
function} a_n^C = - \frac{2}{ \sqrt{z_0^2 - 1} (z_0 \pm \sqrt{z_0^2
- 1})^n }, \quad n \geq 0.
\end{equation*}
Now combining the above two equations, we see that
\begin{equation}
\gamma_n = \frac{c_{n,\frac{1}{2}} \sqrt{z_0^2-1}}{z_0 \pm
\sqrt{z_0^2-1}} {}_2\mathrm{ F}_1\left[\begin{matrix} n + 1,~
\frac{1}{2}; &
\\  n + \frac{3}{2}; &\end{matrix} \hspace{-0.25cm}  \frac{1}{ (z_0 \pm \sqrt{z_0^2 - 1}
)^{2} } \right], \quad n \geq 0.
\end{equation}
This together with Lemma \ref{lemma:asymptotic legendre constant}
and Lemma \ref{lemma:asymptotic of hypergeometric} gives the desired
result. %This proves Theorem \ref{thm:asymptotic ratio for pole
%function}.
\end{proof}

\begin{remark}\label{remark:other singularty}
We remark that the result \eqref{eq:asymptotic of ratio} still holds
if $z_0$ is a singularity of algebraic or logarithmic type. The key
point is that the Chebyshev and Legendre coefficients can be
estimated accurately by evaluating their contour integral
expressions at the singularity multiplied by some common factors.
For example, when $f(x) = (b-x)^{\alpha}$ where $b>1$ and
$\alpha>-1$ is not an integer. From \cite{elliott1964evaluation}, we
have that
\[
a_n^{C} = \frac{\xi(\alpha,b) }{ \sqrt{b^2 - 1}(b + \sqrt{b^2 -
1})^n } \left( 1 + \mathcal{O}(n^{-1}) \right), \quad n
\rightarrow\infty,
\]
where $\xi(\alpha,b)$ is defined by
\[
\xi(\alpha,b) = - \frac{ 2 \sin(\alpha\pi) (b^2 -
1)^{\frac{\alpha+1}{2}} \Gamma(\alpha+1) }{\pi n^{\alpha+1}}.
\]
Adapting the similar arguments and Lemma \ref{lemma:asymptotic
legendre constant} and Lemma \ref{lemma:asymptotic of
hypergeometric} gives
\[
a_n^L = \frac{\xi(\alpha,b) }{ \sqrt{b^2 - 1} (b + \sqrt{b^2 - 1})^n
} \left( \sqrt{n\pi} g(b) + \mathcal{O}(n^{-\frac{1}{2}}) \right),
\quad n \rightarrow\infty,
\]
where $g(b)$ is defined as in \eqref{def:function g}. Combining the
above two estimates, it is easy to check that \eqref{eq:asymptotic
of ratio} still holds; see Figure \ref{fig:comparison exterior
singularity} for numerical illustrations.

%The key point is that those factors due to the type and strength of
%the singularity arise simultaneously in both Legendre and Chebyshev
%coefficients and thus they have little influence on the asymptotic
%behaviour of the ratio $\gamma_n$ for large $n$.

%We remark that the result \eqref{eq:asymptotic of ratio} still holds
%if $z_0$ is a singularity of another kind. The key point is that the
%same factor due to the strength of the singularity arise
%simultaneously in both Legendre and Chebyshev coefficients and thus
%their ratio $\gamma_n$ depends only on the location of the
%singularity. If $z_0$ is a singularity of algebraic type, we refer
%the reader to \cite{elliott1964evaluation} for more discussion.
\end{remark}

\begin{figure}[ht]
\centering
\includegraphics[width=10cm, height = 8cm]{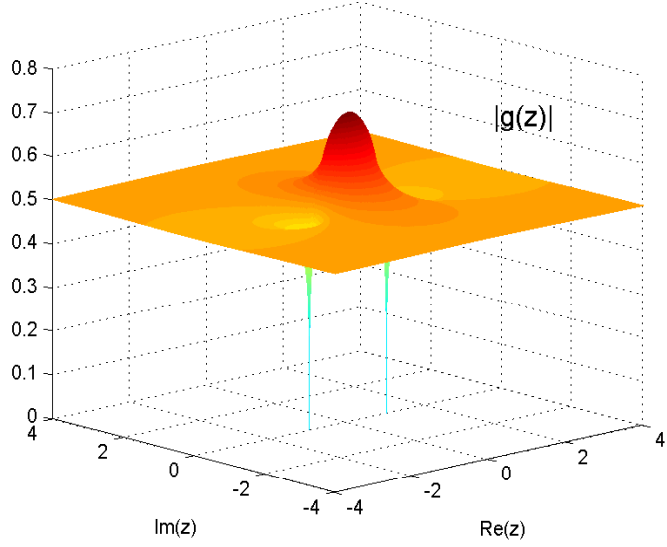}
\caption{Absolute value of the function $g(z)$ in the complex
plane.} \label{fig:absolute value g}
\end{figure}

\begin{remark}\label{remark:property g}
It is easily seen that $g(z) = 0$ when $z = \pm 1$; otherwise, $g(z)
\neq 0$. This leads us to consider that the asymptotic behaviour of
$\gamma_n$ will be greatly different if $z_0$ is an endpoint
singularity. Indeed, we shall show the ratio $\gamma_n$ tends to
some constants as $n$ grows in this case. Moreover, when $g(z_0)
\neq 0$, we can deduce that the Chebyshev coefficient $a_n^C$ decays
approximately $g(z_0) \sqrt{n\pi}$ faster than the Legendre
coefficient $a_n^L$.
\end{remark}

\begin{remark}
It is not difficult to verify that $\lim_{|z|\rightarrow\infty} g(z)
= \frac{1}{2}$ and $g(z) \neq \frac{1}{2}$ for any finite $|z|$,
which implies that the ``Lanczos-Fox-Parker'' proposition is always
false for functions that have a singularity in the complex plane.
Figure \ref{fig:absolute value g} illustrates the absolute value of
$g(z)$ in the complex plane. We can see that $|g(z)|$ converges to
$\frac{1}{2}$ as $|z|\rightarrow\infty$. Meanwhile, we also observe
that $|g(z)|$ attains its maximum value at $z=0$ and $|g(0)| =
\frac{1}{\sqrt{2}}$, which implies that the fastest possible rate of
growth of $\gamma_n$ will be close to $\sqrt{\frac{n\pi}{2}}$ as $n$
grows.
\end{remark}

\begin{remark}
When the pole $z_0$ is real, we have
\begin{align}
g(z_0) = \left\{\begin{array}{ccc}
                                          {\displaystyle  \sqrt{\frac{z_0^2 - 1}{ (z_0 + \sqrt{z_0^2 - 1})^2 - 1 }}  },   & \mbox{$z_0>1$, } \\ [14pt]
                                          {\displaystyle  \sqrt{\frac{z_0^2 - 1}{ (z_0 - \sqrt{z_0^2 - 1})^2 - 1 }}  },   & \mbox{$z_0 <
                                          -1$}.
                                        \end{array}
                                        \right.
\end{align}
It is easy to verify that $g(z_0)$ is monotonically decreasing on
the interval $(-\infty, -1)$ and is monotonically increasing on the
interval $(1, \infty)$ and $ \lim_{z_0 \rightarrow \pm1} g(z_0) = 0$
and $ \lim_{z_0 \rightarrow \pm\infty} g(z_0) = \frac{1}{2}$. If the
pole lies on the imaginary axis, e.g., $z_0 = ib$ where $b$ is a
real and nonzero constant, then
\begin{align}
g(z_0) = g(ib) = \left\{\begin{array}{ccc}
                                          {\displaystyle  \sqrt{\frac{b^2 + 1}{ (b + \sqrt{b^2 +1})^2 + 1 }}  },   & \mbox{$b>0$, }\\ [12pt]
                                          {\displaystyle  \sqrt{\frac{b^2 + 1}{ (b - \sqrt{b^2 +1})^2 + 1 }}  },   & \mbox{$b<0$}.
                                        \end{array}
                                        \right.
\end{align}
It can be seen that $g(z_0)$ is monotonically increasing when $b \in
(-\infty, 0)$ and is monotonically decreasing function when $b \in
(0,\infty)$ and
\begin{equation*}
\lim_{b \rightarrow \pm0 } g(ib) = \frac{1}{\sqrt{2}}, \quad \lim_{b
\rightarrow \pm\infty} g(ib) = \frac{1}{2}.
\end{equation*}
These properties can be confirmed from Figure \ref{fig:absolute
value g}.
\end{remark}

\begin{figure}[ht]
\centering
\includegraphics[width=6.8cm]{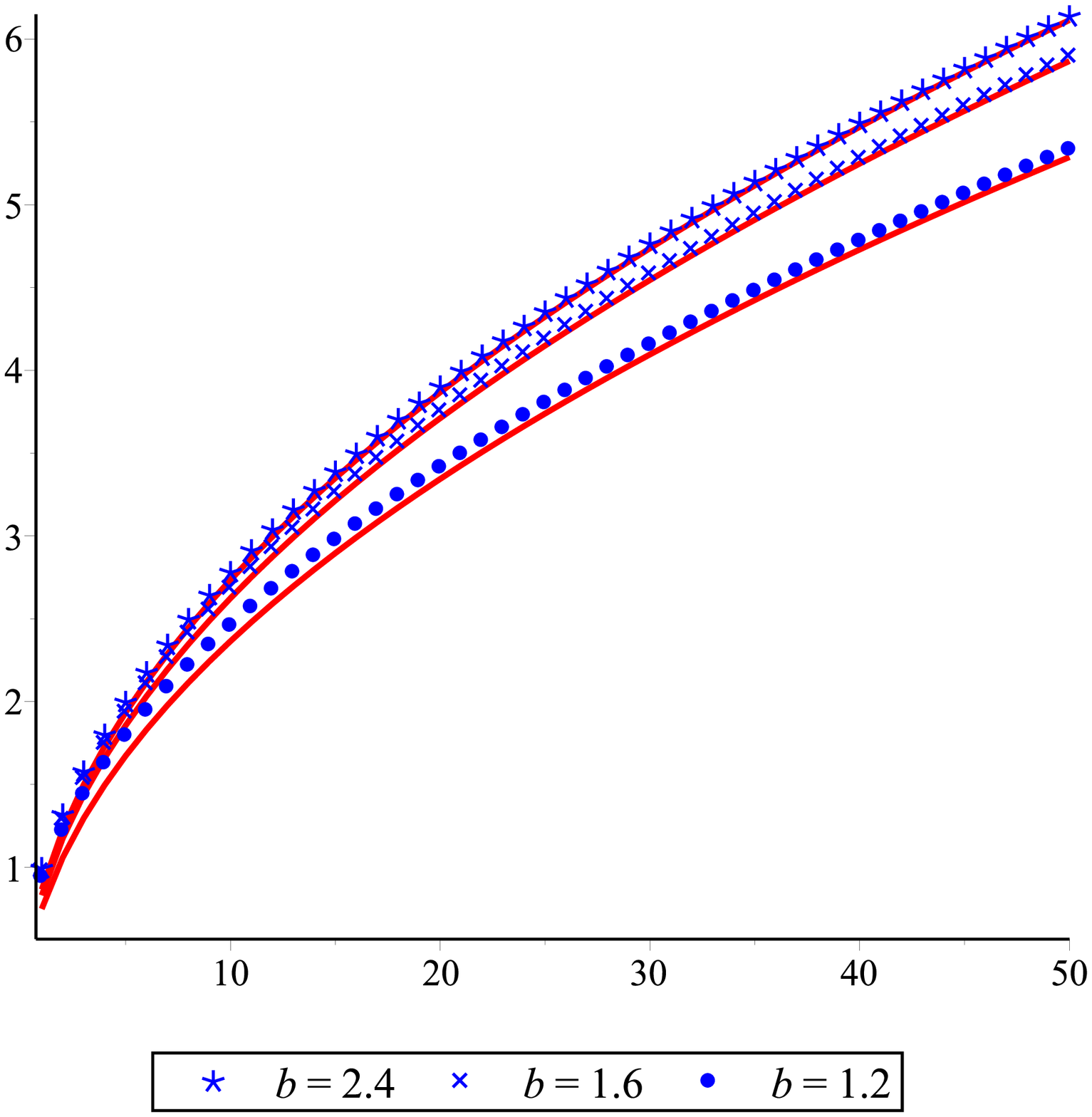}\qquad
\includegraphics[width=6.8cm]{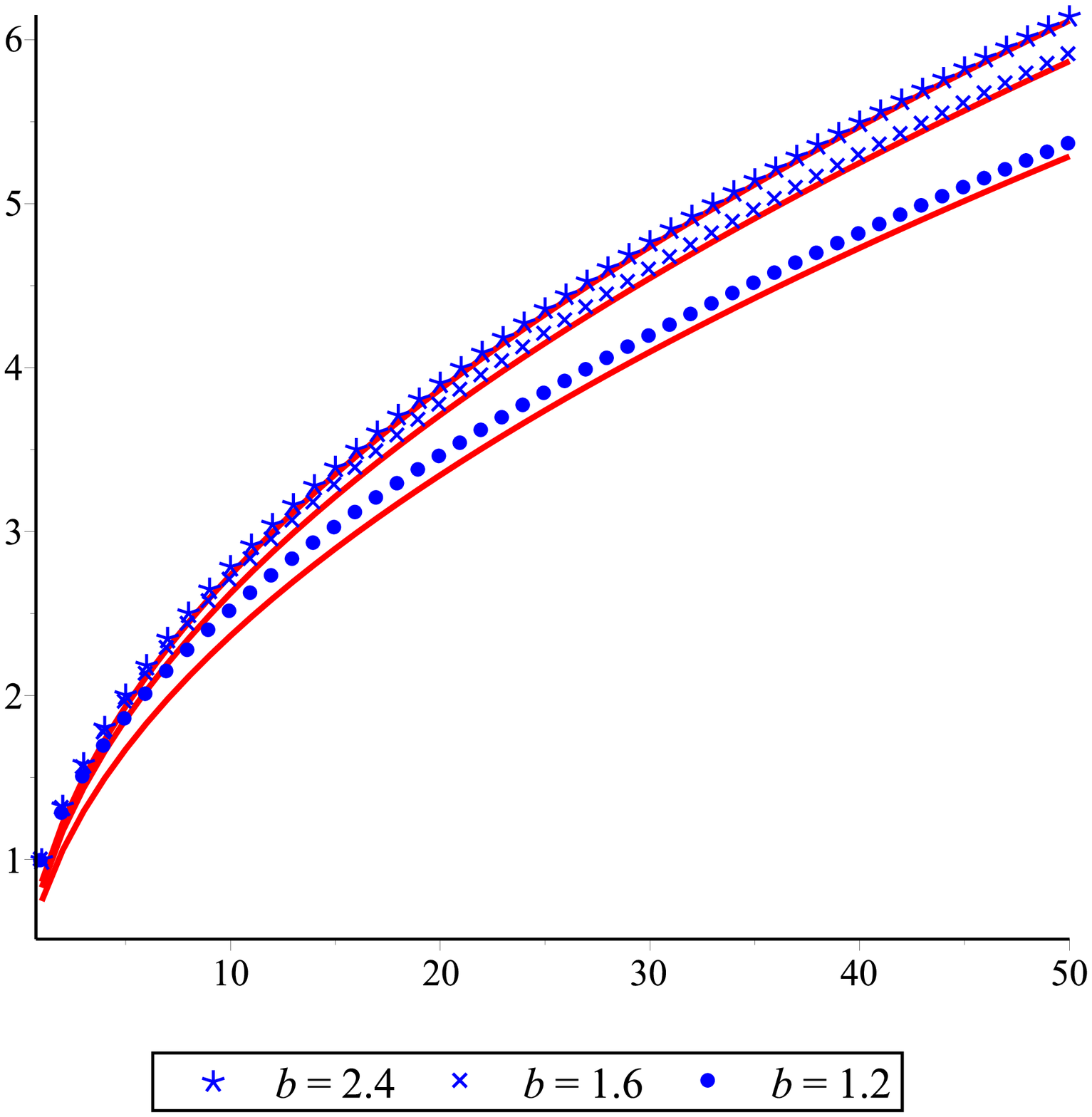}
\caption{The exact values of ratio $\gamma_n$ (dots, asterisks and
crosses) and the predicted values $g(b)\sqrt{n\pi}$ (lines) for
$f(x) = \log(b - x)$ (left) and $f(x) = (b - x)^{\frac{2}{3}}$
(right). Here $n$ ranges from $1$ to $50$.} \label{fig:comparison
exterior singularity}
\end{figure}

\begin{example}
For model functions such as
\[
f(x) = \log(b \pm x),~ (b\pm x)^{\alpha},
\]
where $b>1$ and $\alpha$ is not a nonegative integer. It follows
from Theorem \ref{thm:asymptotic ratio for pole function} and Remark
\ref{remark:other singularty} that
\begin{align*}
\gamma_n = g(b) \sqrt{n\pi} + \mathcal{O}(n^{-\frac{1}{2}}), \quad
n\rightarrow\infty,
\end{align*}
where $g(b)$ is given by \eqref{def:function g}, e.g.,
\begin{align}\label{eq:ratio real pole}
g(b) = \sqrt{\frac{b^2 - 1}{ (b + \sqrt{b^2 - 1})^2 - 1 }}.
\end{align}

In Figure \ref{fig:comparison exterior singularity}, we illustrate
the exact values of $\gamma_n$ and the predicted values
$g(b)\sqrt{n\pi}$ for two functions $f(x) = \log(b-x)$ and $f(x) =
(b - x)^{\frac{2}{3}}$, where $g(b)$ is defined by \eqref{eq:ratio
real pole}. For each function, we show numerical results for three
values of $b$. As expected, the asymptotic behaviour of $\gamma_n$
is in excellent agreement with our predicted values.
\end{example}

%\begin{example}
%Our results can be extended to  As an illustrative example, we
%consider the well known Runge function
%\[
%f(x) = \frac{1}{x^2 + b^2},
%\]
%where $b$ is real and $b\neq0$. This function has two poles at $x =
%\pm b i$ and their residues are $\pm \frac{1}{2 b i}$, respectively.
%
%
%Since $f$ is even, we have $a_n^{L}$ and $a_n^C$ are zero for odd
%$n$. Hence, we only need to consider the case where $n$ is even. In
%analogy to \eqref{eq:real pole function}, its exact Gegenbauer
%coefficients are given explicitly by
%\begin{equation}\label{eq:runge gegenbauer coefficients}
%a_{n}^{\lambda} = (-1)^{\frac{n}{2}} \frac{ 2c_{n,\lambda} }{ b (b +
%\sqrt{b^2 + 1})^{n+1} } {}_2\mathrm{ F}_1\left[\begin{matrix} n + 1,
%&  1 - \lambda \\  n + \lambda + 1, \hspace{-1cm} &\end{matrix} ;
%\frac{-1}{ (b + \sqrt{b^2 + 1} )^{2} } \right].
%\end{equation}

\subsection{The case $z_0 = \pm 1$}
In this subsection we consider the case that $z_0$ is an endpoint
singularity. Boyd and Petschek in \cite{boyd2013relationships}
considered the special case $f(x) = (1 - x)^{\alpha}$ where
$\alpha>0$ is not an integer and showed that the decay rate of the
Chebyshev coefficient $a_n^C$ is the same as that of the Legendre
coefficient $a_n^L$. In what follows, we consider the asymptotic
behaviour of $\gamma_n$ for the model functions $f(x) = \log(1\pm
x)$ and $f(x) = (1 \pm x)^{\alpha}$ where $\Re(\alpha) >
-\frac{1}{2}$ is not an integer. We present delicate results on the
asymptotic behaviour of $\gamma_n$.

%Note that these functions we
%consider here are more general than that considered in
%\cite{boyd2013relationships}.

\begin{theorem}\label{thm:asymptotic gamma endpoint singularity}
For the functions $f(x) = (1\pm x)^{\alpha}$ where $\Re(\alpha)>
-\frac{1}{2}$ is not an integer, we have
\begin{align}\label{eq:asymptotic gamma algebraic endpoint}
\gamma_n = \frac{\Gamma(\alpha+1)}{\Gamma(\alpha+\frac{1}{2})}
\sqrt{\pi} + \mathcal{O}(n^{-1}), \quad n\rightarrow\infty.
\end{align}
For the functions $f(x) = \log(1 \pm x)$, we have
\begin{align}\label{eq:asymptotic gamma logarithmic endpoint}
\gamma_n = 1 + \mathcal{O}(n^{-1}), \quad n \rightarrow\infty.
\end{align}
\end{theorem}
\begin{proof}
We consider the case $f(x) = (1 - x)^{\alpha}$. Since the Legendre
and Chebyshev coefficients correspond to the special case $\lambda =
\frac{1}{2}$ and $\lambda = 0$, respectively, we have from Lemma
\ref{lemma:legendre and chebyshev} and Lemma \ref{lemma:gegenbauer
algebraic endpoint function} that
\[
a_n^L = (-1)^n \frac{ 2^{\alpha+1} \Gamma(\alpha+1)^2
(n+\frac{1}{2}) }{ \Gamma(\alpha-n+1) \Gamma(\alpha+n+2) },
\]
where $\Re(\alpha) > -1$, and
\[
a_n^C = (-1)^n \frac{ 2^{\alpha+1} \Gamma(\alpha+\frac{1}{2})
\Gamma(\alpha+1) }{\sqrt{\pi} \Gamma(\alpha-n+1) \Gamma(\alpha+n+1)
},
\]
where $\Re(\alpha) > -\frac{1}{2}$. Thus, when $\Re(\alpha) >
-\frac{1}{2}$, combining the above two equations, we get
\[
\gamma_n = \left(n+\frac{1}{2}\right) \frac{ \Gamma(n+\alpha+1)
}{\Gamma(n+\alpha+2)}
\frac{\Gamma(\alpha+1)}{\Gamma(\alpha+\frac{1}{2})} \sqrt{\pi},~~
n\geq 0.
\]
By using \eqref{eq:asymptotic ratio gamma function}, the result
\eqref{eq:asymptotic gamma algebraic endpoint} follows. For the case
$f(x) = \log(1 - x)$, using Lemma \ref{lemma:legendre and chebyshev}
and Lemma \ref{lemma:gegenbauer endpoint logarithm function} we see
that
\[
a_n^L = -\frac{2n+1}{n(n+1)}, \quad a_n^C = -\frac{2}{n},
\]
for each $n \geq 1$. Hence the result \eqref{eq:asymptotic gamma
logarithmic endpoint} follows. The other cases $f(x) =
(1+x)^{\alpha},~ \log(1+x)$ can be proved similarly and we omit the
details. This proves Theorem \ref{thm:asymptotic gamma endpoint
singularity}.
\end{proof}

\begin{remark}
For functions of the form $f(x) = (1\pm x)^{\alpha}$, where
$\Re(\alpha)>-\frac{1}{2}$ is not an integer, Theorem
\ref{thm:asymptotic gamma endpoint singularity} implies that the
Chebyshev coefficient $a_n^C$ decays
$\frac{\Gamma(\alpha+1)}{\Gamma(\alpha+\frac{1}{2})} \sqrt{\pi}$
faster than its Legendre counterpart $a_n^L$. For functions of the
form $f(x) = \log(1\pm x)$, however, both coefficients are almost
the same for large $n$.
\end{remark}

\begin{remark}
For more general functions such as $f(x) = (1\pm x)^{\alpha} h(x)$
and $f(x) = \log(1\pm x) h(x)$ where $\Re(\alpha) > -\frac{1}{2}$ is
not an integer and $h(x)$ is analytic in a neighborhood of the
interval $[-1,1]$, note that the main contributions to the Chebyshev
and Legendre coefficients come from the endpoint singularity, it is
not difficult to verify that the results of Theorem
\ref{thm:asymptotic gamma endpoint singularity} still hold.
\end{remark}

\subsection{The case $z_0 \in (-1,1)$}
When the singularity $z_0 \in (-1,1)$, we are not able to derive the
asymptotic behaviour of their Chebyshev and Legendre coefficients,
since in this case $f(z)$ is not analytic and their contour integral
expressions can not be used.

The special case $f(x) = |x|$ has been analyzed in
\cite{boyd2013relationships} and it was shown that
$\gamma_n/\sqrt{n\pi} \sim \frac{1}{\sqrt{2}}$ as
$n\rightarrow\infty$. However, the situation becomes much more
complicated when $z_0 \neq 0$. As an illustrative example, we show
the behaviour of $\gamma_n/\sqrt{n\pi}$ for $f(x) =
\log|x-\frac{1}{4}|$ and $f(x) = |x-\frac{1}{4}|^{\frac{1}{2}}$ in
Figure \ref{fig:comparison interior singularity}. It is clear to see
that $\gamma_n/\sqrt{n\pi}$ oscillates around some finite values. On
the other hand, we can also see that $\gamma_n/\sqrt{n\pi} =
\mathcal{O}(1)$, which implies that $a_n^C$ decays
$\mathcal{O}(\sqrt{n})$ faster than $a_n^L$. However, a precise
result on the asymptotic behaviour of $\gamma_n$ is still open.

\begin{figure}[ht]
\centering
\includegraphics[width=6.8cm]{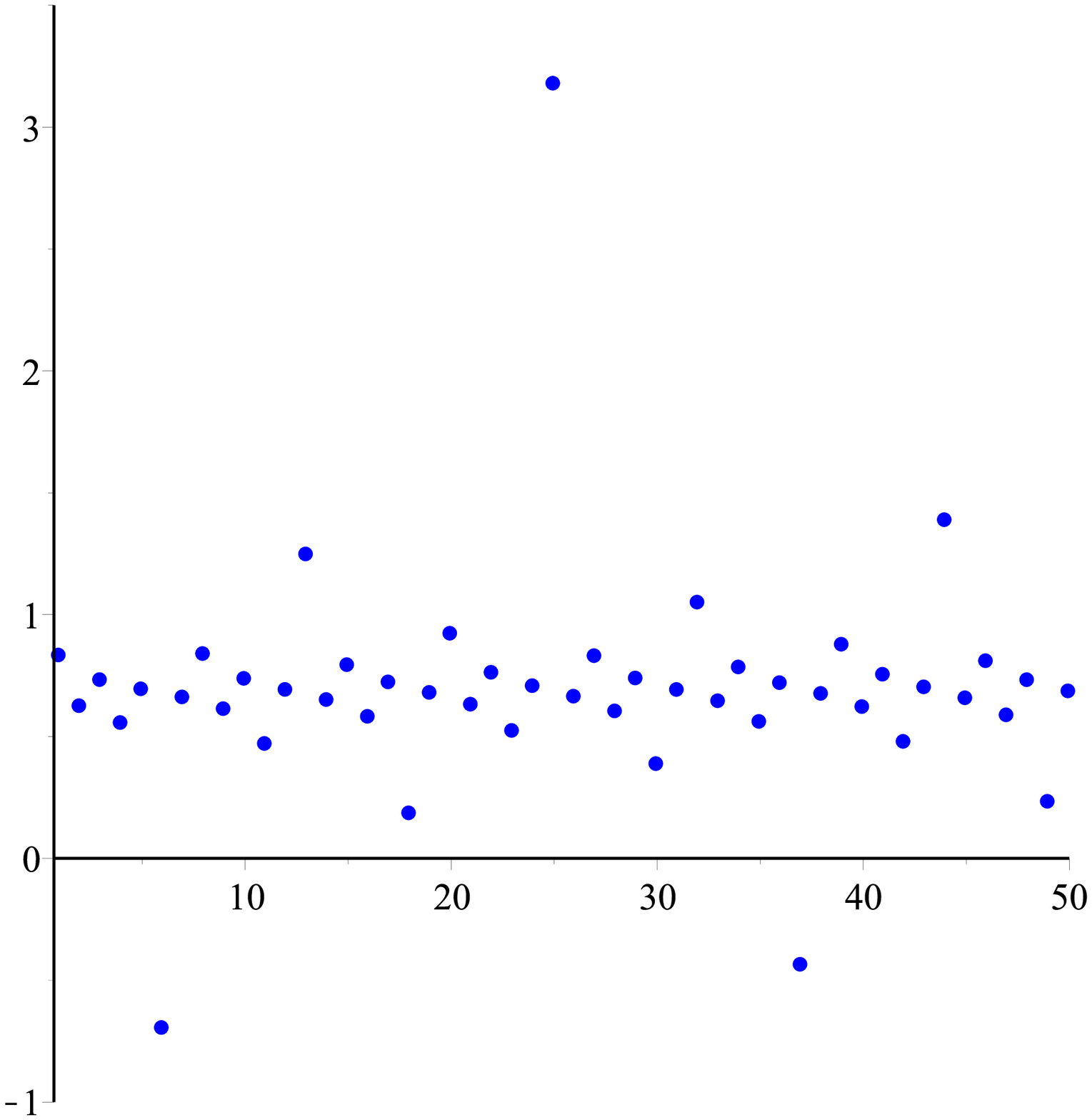}\qquad
\includegraphics[width=6.8cm]{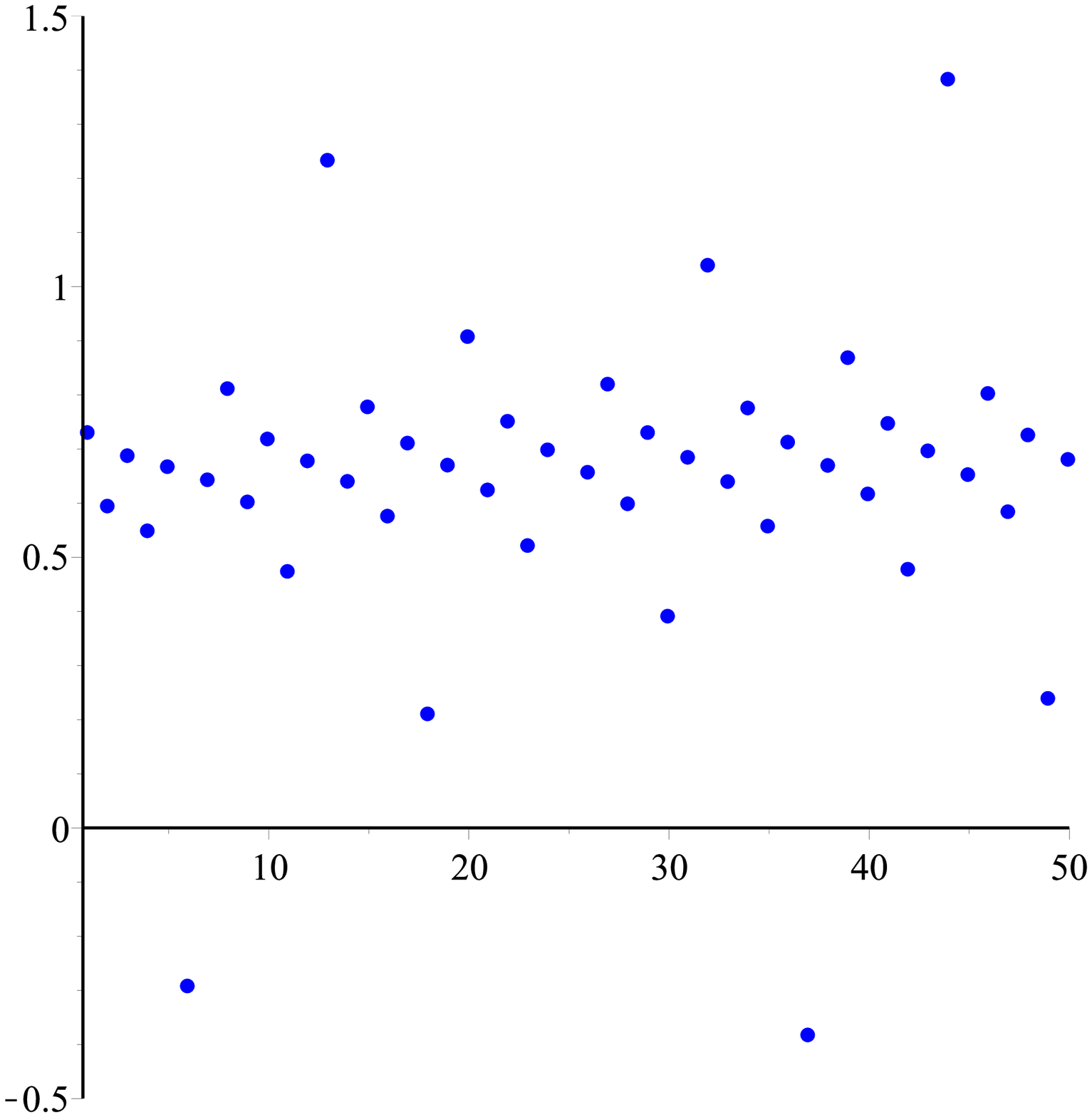}
\caption{The values of $\frac{\gamma_n}{\sqrt{n\pi}}$ for $f(x) =
\log|x - \frac{1}{4}|$ (left) and $f(x) = |x -
\frac{1}{4}|^{\frac{1}{2}}$ (right). Here $n$ ranges from $1$ to
$50$.} \label{fig:comparison interior singularity}
\end{figure}

\section{Concluding remarks}\label{sec:conclusion}
Gegenbauer expansion is an important tool in the resolution of Gibbs
phenomenon and the numerical solution of differential equations. In
this work, we have proposed a simple derivation of the contour
integral expression of the coefficients in the Gegenbauer series
expansion. We have derived optimal and explicit estimates for the
Gegenbauer coefficients and these estimates are sharper than the
existing ones. We further apply these optimal estimates to establish
some rigorous and computable bounds for the truncated Gegenbauer
series in the uniform norm. Additionally, we also consider the
comparison of the decay rates of the Legendre and Chebyshev
coefficients. Delicate results on the asymptotic behaviour of the
ratio of the Legendre coefficient $a_n^L$ to the Chebyshev
coefficient $a_n^C$ are presented.

We point out that our results can be easily extended to the
comparison of spectral methods based on Chebyshev and Legendre
expansions. It is well known that the maximum error of the truncated
Chebyshev expansion can be estimated approximately by the absolute
value of the first neglected term if the Chebyshev coefficients
decay rapidly; see \cite{elliott1965truncation}. This implies the
comparison of spectral methods using Chebyshev and Lgendre
expansions can be approximately transformed into the comparison of
the corresponding Legendre and Chebyshev coefficients. For example,
if the singularity of $f$ is outside the interval $[-1,1]$ which
implies that the Legendre and Chebyshev coefficients decay
exponentially, then we can deduce that the rate of convergence of
the spectral method using Chebyshev expansion is
$\mathcal{O}(\sqrt{n})$ faster than that of its Legendre
counterpart, where $n$ is the number of terms in both expansions.

\section*{Acknowledgement}
The author would like to thank the anonymous referees for their
valuable comments.

\appendix

\section[Appendix]{Gegenbauer expansion coefficients of functions with endpoint singularities}\label{append:Gegenbauer coefficient endpoint sinfularity}
We present explicit formulas of the Gegenbauer expansion
coefficients for $f(x) = \log(1 \pm x)$ and $f(x) = (1 \pm
x)^{\alpha}$ where $\Re(\alpha) > -\frac{1}{2}$ is not an integer.

\begin{lemma}\label{lemma:gegenbauer algebraic endpoint function}
For the functions $f(x) = (1 \pm x)^{\alpha}$ where $\Re(\alpha)
> -\frac{1}{2}$ is not an integer. Then, for $\Re(\lambda)> -\frac{1}{2}$ and $\Re(\lambda + \alpha)> -\frac{1}{2}$,
\begin{align}\label{eq:gegenbauer coefficients algebraic endpoint}
a_n^{\lambda} = \mu_n \frac{ 2^{4\lambda+\alpha-1} \Gamma(\lambda)^2
\Gamma(\alpha+\lambda +\frac{1}{2}) \Gamma(\lambda+\frac{1}{2})
\Gamma(\alpha+1) (n+\lambda) }{\pi \Gamma(2\lambda)
\Gamma(\alpha-n+1) \Gamma(\alpha+2\lambda+n+1)},~~ n\geq0,
\end{align}
where $\mu_n = (-1)^{n}$ when $f(x) = (1-x)^{\alpha}$ and $\mu_n =
1$ when $f(x) = (1+x)^{\alpha}$.
\end{lemma}

\begin{proof}
We only consider the case $f(x) = (1-x)^{\alpha}$ since the argument
in the case $f(x) = (1+x)^{\alpha}$ is similar. For each $n \geq 0$,
using the definition of $a_n^{\lambda}$ we have
\begin{align*}
a_n^{\lambda}  = \frac{1}{h_n^{(\lambda)}} \int_{-1}^{1}
(1-x)^{\alpha+\lambda-\frac{1}{2}} (1+x)^{\lambda-\frac{1}{2}}
C_n^{(\lambda)}(x) dx.
\end{align*}
Using \cite[Eqn.~7.311]{gradshteyn2007ryzhik} gives
\eqref{eq:gegenbauer coefficients algebraic endpoint}. This
completes the proof.
\end{proof}

\begin{lemma}\label{lemma:gegenbauer endpoint logarithm function}
For the functions $f(x) = \log(1 \pm x)$. For each $n \geq 1$, their
Gegenbauer coefficients can be written explicitly by
\begin{align}\label{eq:gegenbauer coefficient endpoint logarithmic}
a_n^{\lambda}  =  \mu_n \frac{4^{\lambda} \Gamma(\lambda)
\Gamma\left( \lambda + \frac{1}{2} \right) }{\sqrt{\pi}} \frac{
\Gamma(n+1) \Gamma(n+\lambda+1) }{ n \Gamma(n+\lambda) \Gamma(n+1 +
2\lambda) },
\end{align}
where $\mu_n = -1$ when $f(x) = \log(1-x)$ and $\mu_n = (-1)^{n+1}$
when $f(x) = \log(1+x)$.
\end{lemma}
%We start our analysis from the case $b = 1$.
%\begin{lemma}\label{lemma:gegenbauer endpoint logarithm function}
%For the function $f(x) = \log(1 - x)$. For each $n \geq 1$, its
%Gegenbauer coefficients can be written explicitly by
%\begin{align}\label{eq:gegenbauer coefficient endpoint logarithmic}
%a_n^{\lambda}  = - \frac{4^{\lambda} \Gamma(\lambda) \Gamma\left(
%\lambda + \frac{1}{2} \right) }{\sqrt{\pi}} \frac{ \Gamma(n+1)
%\Gamma(n+\lambda+1) }{ n \Gamma(n+\lambda) \Gamma(n+1 + 2\lambda) }.
%\end{align}
%For the case $f(x) = \log(1+x)$, then for each $n \geq 1$, we have
%\begin{align}
%a_n^{\lambda}  = (-1)^{n+1} \frac{4^{\lambda} \Gamma(\lambda)
%\Gamma\left( \lambda + \frac{1}{2} \right) }{\sqrt{\pi}} \frac{
%\Gamma(n+1) \Gamma(n+\lambda+1) }{ n \Gamma(n+\lambda) \Gamma(n+1 +
%2\lambda) }.
%\end{align}
%\end{lemma}
\begin{proof}
We first consider the case $f(x) = \log(1-x)$. Following the ideas
exposed in \cite{elliott1964evaluation}, the contour integral of the
$a_n^{\lambda}$ can be deformed as an ellipse $\mathcal{E}_{R}$ in
the positive direction and a small circle with center at $z = 1$ in
the negative direction. Meanwhile, these two contours are connected
by two line segments which are parallel and above and below the real
axis. Letting $R\rightarrow\infty$ and the radius of the circle tend
to zero, we note that the integral along $\mathcal{E}_{R}$ vanishes
for each $n\geq1$ and the integral along the circle vanishes. Thus,
combining the contributions to $a_n^{\lambda}$ from these two line
segments we find
\begin{align*}
a_n^{\lambda} &= - 2 c_{n,\lambda} \int_{1}^{\infty} \frac{1}{(t +
\sqrt{t^2 - 1})^{n+1}} {}_2\mathrm{ F}_1\left[\begin{matrix} n + 1, & 1 - \lambda; \\
n + \lambda + 1; \hspace{-.7cm} &\end{matrix}  \frac{1}{ (t +
\sqrt{t^2 - 1} )^{2} } \right] dt. %\nonumber \\
%& = - c_{n,\lambda} \sum_{m=0}^{\infty} \frac{ (n+1)_{m} (1 -
%\lambda)_m }{ (n + \lambda + 1)_m } \left[ \frac{1}{n+2m} -
%\frac{1}{n+2m+2} \right] \nonumber \\
%& = - \frac{4^{\lambda} \Gamma(\lambda) \Gamma\left( \lambda +
%\frac{1}{2} \right) }{\sqrt{\pi}} \frac{ \Gamma(n+1)
%\Gamma(n+\lambda+1) }{ n \Gamma(n+\lambda) \Gamma(n+1 + 2\lambda) },
%\quad n \geq 1.
\end{align*}
Setting $t = \cosh\theta$ and using some elementary calculations
gives \eqref{eq:gegenbauer coefficient endpoint logarithmic}. The
case $f(x) = \log(1+x)$ can be handled similarly and we omit the
details. This completes the proof.
\end{proof}

\bibliographystyle{abbrv}
\bibliography{legendre}

%\begin{thebibliography}{99}

%\end{thebibliography}

\end{document}